\documentclass[12pt]{amsart}
\usepackage{amsfonts,amsmath,amsthm,oldgerm,amssymb,amscd,}
\usepackage{setspace}
\usepackage[all,cmtip]{xy}

\usepackage{fancyhdr}
\usepackage{psfrag}
\usepackage{graphicx}

\usepackage[indent]{caption}
\usepackage{indentfirst}
\usepackage{float}
\usepackage{latexsym}
\usepackage{graphics}
\usepackage{verbatim}

\newcommand{\Q}{{\mathbb Q}}
\newcommand{\R}{{\mathbb R}}
\newcommand{\Z}{{\mathbb Z}}
\newcommand{\C}{{\mathbb C}}

\newcommand{\fg}{{\mathfrak g}}
\newcommand{\fh}{{\mathfrak h}}
\newcommand{\cR}{{\mathcal R}}
\renewcommand{\O}{{\mathcal O}}

\newtheorem{theorem}{Theorem}
\newtheorem{lemma}{Lemma}
\newtheorem{corollary}{Corollary}
\newtheorem{proposition}[theorem]{Proposition}

\theoremstyle{definition}

\newtheorem{example}{Example}
\newtheorem{definition}{Definition}
\numberwithin{definition}{section}

\numberwithin{equation}{section}

\theoremstyle{remark}
\newtheorem{remark}{Remark}

\begin{document}

\title[Irreducibility of irreducible characters]{On
  the irreducibility of irreducible characters\\ of simple Lie algebras}

\author{C.~S.~Rajan}

\address{Tata Institute of Fundamental  Research, Homi Bhabha Road,
Bombay - 400 005, INDIA.}  \email{rajan@math.tifr.res.in}


\begin{abstract}
We establish an irreducibility property for the characters of finite
dimensional, irreducible representations of simple Lie algebras (or
simple algebraic groups) over the complex numbers, i.e., that  the
characters of irreducible representations  are irreducible after
dividing out by (generalized) Weyl denominator type factors.

For $SL(r)$ the irreducibility result is the following: let
$\lambda=(a_1\geq a_2\geq \cdots a_{r-1}\geq 0)$ be the highest weight of an
irreducible rational representation $V_{\lambda}$ of $SL(r)$.
 Assume that the integers $a_1+r-1, ~a_2+r-2,
\cdots, a_{r-1}+1$ are relatively prime. Then the character
$\chi_{\lambda}$ of $V_{\lambda}$ is strongly irreducible in the
following sense: for any natural number $d$, the function
$\chi_{\lambda}(g^d), ~g\in SL(r,\C)$ is irreducible in the ring of
regular functions of $SL(r,\C)$.
\end{abstract}

\maketitle

\section{Introduction}
In \cite{R}, the following unique factorization
property of tensor products
of irreducible representations of a complex simple Lie algebra $\frak
g$  was proved:

\begin{theorem} \label{thm:tensor}
Let  $V_1, \cdots, V_n, ~W_1, \cdots, W_m$ be non-trivial irreducible
representations of $\frak g$ (resp. rational irreducible
 representations of $GL(r)$ for $r\geq 2$ with trivial determinant ) such
that
\[ V_1\otimes \cdots\otimes V_n\simeq W_1\otimes
\cdots\otimes W_m,\] as $\frak g$ (resp. $GL(r)$)-modules.
Then $n=m$, and there is a
permutation $\sigma$ of $\{1,\cdots, n\}$ such that $V_i\simeq
W_{\sigma(i)}$.
\end{theorem}

In this introduction we restrict ourselves to $GL(r)$.
Given a rational representation $V$ of $GL(r)$, let
\[ \chi_V(g)={\rm Tr}(V(g)), \quad g\in GL(r), \]
denote the character of $V$.  Since the character determines the representation
upto isomorphism, the hypothesis of the foregoing theorem can be
reformulated as an equality of products of characters,
\begin{equation}\label{prodchar}
 \chi_{V_1} \cdots \chi_{V_n}= \chi_{W_1} \cdots\chi_{W_m}.
\end{equation}
Now, the character is  a regular $GL(r)$-invariant
function of $GL(r)$. The unique decomposition  of tensor products
will follow if the non-trivial characters
 of irreducible rational representations of $GL(r), ~r\geq 2$
are irreducible in the ring of invariant regular functions on
$GL(r)$.

But this turns out to be manifestly false for $GL(2)$, where the characters are
given by cyclotomic type
homogeneous polynomials in two variables of the
form $(x_1^{a+1}-x_2^{a+1})/(x_1-x_2)$;
hence factorizable over $\C$.

 Upto twisting by a power of the
determinant, we can assume that the irreducible representations
$V_{\lambda}$  of $GL(r)$ are parametrized by their
highest weights,
\[ \lambda =(a_1, \cdots a_r), ~~a_1\geq
a_2  \geq \cdots\geq  a_r=0,\]
where $a_i$ are nonnegative integers.
Let $T$ denote the diagonal torus of $GL(r)$ consisting of
diagonal matrices with diagonal entries given by $x=(x_1,\cdots,
x_r)$. Let $W$ denote the Weyl group of $(GL(r), T)$,
the  symmetric group on $r$-variables, acting by
permutations on $T$.
Restricting to the torus $T$, gives an isomorphism of the algebra
of conjugacy invariant regular functions on $GL(r)$ onto the algebra
of Weyl group invariant
regular functions on $T$; hence we can consider the characters
restricted to $T$.
Let $\epsilon:~W\to \Z/2\Z$ be the sign homomorphism.
For a weight  $\mu=(b_1,\cdots, b_r)$ with $b_1, \cdots, b_r$ non-negative
integers, define the Schur-Weyl function $S({\mu})(x)$ by,
\[ S(\mu)(x)=\sum_{\sigma\in W}\epsilon(\sigma)
x_1^{b_{\sigma(1)}}\cdots
x_r^{b_{\sigma(r)}}=\mbox{det}(x_i^{b_j}).  \]
This is an alternating polynomial in variables $x_1, \cdots, x_r$,
 which vanishes if
$b_i=b_j$ for a pair of indices $i\neq j$.
Let
\[\rho=(r-1, r-2, \cdots, 0)\]
be the `Weyl weight'.  The Weyl denominator
function $S(\rho)(x)$ is the Vandermonde determinant and  has a product
decomposition,
\[ S(\rho)(x)=\prod_{i<j}(x_i-x_j).\]
Given any weight $\mu$ of $GL(r)$, it is easy to observe that  the Weyl
denominator $S(\rho)$ divides $S(\mu)$
in the polynomial ring in $r$-variables.
The Schur-Weyl character formula for $GL(r)$ gives the restriction of
the character $\chi_{\lambda}$ of $V_{\lambda}$ to the torus $T$  by the
formula,
\begin{equation}\label{eqn:wcf}
\begin{split}
\mbox{Weyl character formula}:
\chi_{\lambda}(x)&=S(\lambda+\rho)(x)/S(\rho)(x)\\
 &= \mbox{det}(x_i^{a_j+r-j})/\mbox{det}(x_i^{r-j})\quad x\in
 T.
\end{split}
\end{equation}

Based on the factorization of characters for $GL(2)$ governed by
cyclotomic theory,
the initial approach in the general case was
to look for divisibility relations amongst
the irreducible characters parametrized by appropriate
`divisibility properties' amongst the highest weights.
Now let  $\lambda =
(a_1>\cdots  a_{r-1}> a_r=0)$ be a dominant regular  weight for
$GL(r)$.
Denote by $d(\lambda)$ the greatest common divisor of the integers
$a_i$:
\[ d(\lambda):=\mbox{g.c.d.}(a_1, \cdots, a_r).\]
 We observe that  the Schur-Weyl sum
satisfies  the  `scaling' relation,
\begin{equation}
S(d\lambda)(x_1,\cdots, x_r)=S(\lambda)(x_1^d,\cdots, x_r^d),
\end{equation}
where $d\lambda=(da_1, \cdots, da_r)$.
Combined with the divisibility relation
$S(\rho)|S(\lambda)$ for any weight $\lambda$,
 this implies  the divisibility relation,
\begin{equation}\label{eqn:divisibility}
 S(d\rho)|S(\lambda) \quad \mbox{if}\quad
  d|d(\lambda).
\end{equation}

After  a few calculations for $GL(3)$,
it turns out that apart from the
obvious divisibilty relation given by  Equation
(\ref{eqn:divisibility}),
other divisibility relations are hard to come
by.  If $\lambda$ is not a multiple of $\rho$,
define $C(\lambda)$ as the quotient,
\[C(\lambda)=S(\lambda)/S(d(\lambda)\rho).\]
This defines a symmetric polynomial in $r$-variables.
The experimental observations for $GL(3)$ leads us to expect
and prove the following theorem:

\begin{theorem}\label{thm:glrirr} Let $\lambda =(a_1, \cdots a_r),
~~a_1> a_2 > \cdots > a_r=0$ be a dominant, regular integral weight
for  $GL_r$.  If $\lambda$ is not a multiple of $\rho$ viz.,
$\lambda\neq d(\lambda)\rho$, then $C(\lambda)$ is absolutely
irreducible, i.e., irreducible  in the ring  $\C[x_1, \cdots, x_r]$.
\end{theorem}

This theorem has been proved R. Dvornicich and U. Zannier \cite{DZ}.  However
their motivation and proof are completely different.

An uniqueness result can also be established (see Theorem
\ref{thm:unique}), that $\lambda$ can be recovered from $C(\lambda)$
as long as $C(\lambda)\neq 1$. As a corollary one gets a proof of
Theorem \ref{thm:tensor}.

We now give a brief indication of  the  proof of Theorem
\ref{thm:glrirr}, and refer to Section \ref{sec:keyprops} for further
details.  The proof  proceeds by induction on $r$. The inductive step
is carried out using
the cofactor expansion  of the determinant expression for
$S(\lambda)$.
The cofactor expansion gives us a polynomial in one
variable with coefficients that are again of the form $S(\mu)$ with
$\mu$ a weight for $GL(r-1)$.

There are three parts to the proof. The heart of the proof is the
following (see Proposition \ref{prop:key}):
suppose $\mu, ~\eta$ are dominant regular weights for
$GL(r-1)$ with $\eta=\mu+(c,0, \cdots, 0)$ for some natural number
$c$,  such that  both $d(\mu)$
and $d(\eta)$ are divisible by a natural number $d$. Assume further that
there exists non-monomial symmetric polynomials $U, ~V$  and polynomials $X, ~Y$
satisfying  the following system of equations:
\[ UV=S(\mu)/S(d\rho) \quad \text{and} \quad UX+VY=S(\eta)/S(d\rho).\]
Then the conclusion is that $(d(\mu), d(\eta))>d$. The proof
of this proposition uses some arithmetical ideas when $r=3$. For
higher ranks, the key observation is to realize that an inductive
proof is possible and this in turn depends crucially on the fact
that the Weyl denominator $S(\rho)$ divides any $S(\lambda)$.

The rest of the proof is built around this proposition. An
use of Eisenstein criterion allows us to  rule out symmetric
`monic' factorizations, i.e., symmetric factorizations
$C(\lambda)=U_0V_0$ such that either both the leading coefficients of
$U_0$ and $V_0$ (with respect to the `cofactor expansion'
expressing them as a polynomial in
$x_1$ with coefficients polynomials in the $r-1$ variables $x_2,
\cdots, x_r$) or their
constant coefficients are non-monomial (see Proposition
\ref{prop:nonmonic} and Proposition \ref{prop:nonmonicglr}).

Finally, one reduces a non-symmetric factorization to either of the above
propositions.

The outline of this paper is as follows: in the next section, we state
the theorem for a general simple Lie algebra, where we consider the
characters as elements in the algebra of the Weyl group invariants of
the group algebra of the weight lattice.  In Section
\ref{sec:groups}, the irreducibility theorem is stated in the context
of regular functions of a simple algebraic group, and a proof is given
assuming the statement for the Lie algebra. In Section
\ref{sec:cofactor} we recall the technique of cofactor expansions,
which allows an inductive set up based on the rank for the proof of
the irreducibility result.  Section
\ref{sec:keyprops} contains the statements of the key propositions and
a proof of the main Theorem \ref{thm:main} assuming the
validity of these propositions.

The overall idea of the proof is the
same as that of $GL(r)$; in fact, it is simpler in some places for the
Lie algebras of type $D$ and $E$. However, both the statement and the
proofs of the various propositions are more complicated for the non-simply laced
Lie algebras.

Section \ref{sec:sl2} gives a proof of
the key Proposition \ref{prop:key} for the root system $sl_2$
using some facts from arithmetic. The whole proof emanates from this
proposition, and how it helps in establishing the irreducibility
property for $GL(3)$ (the interested reader can read this section
first, especially Remark \ref{rmk:idea}).

Section \ref{sec:pfkeyprop} gives a proof of Proposition \ref{prop:key}
in the general case by an inductive argument, a `neat swindle' using
the universal divisibility of the Weyl denominator.

The proof of Proposition
\ref{prop:nonmonic} ruling out the existence of invariant monic
factorizations is given in Section
\ref{sec:nonmonic}. This is an application of the method of proof
used in the classical Eisenstein criterion for irreducibility of
polynomials.

The uniqueness result Theorem \ref{thm:unique} is
proved in Section \ref{sec:unique}. This section also contains a
preliminary result used in the proofs of the various propositions: the
fact that $S(\lambda)$ is separable, say for $GL(s)$ with $s<r$.
In this case, it  is quite easy to observe assuming Theorem
\ref{thm:glrirr} and the factorization of $S(d\rho)$.

 Finally in Section
\ref{sec:noninv} we extend the proof of the irreducibility result for
$GL(r)$ from
the ring of symmetric functions in $r$-variables to the  polynomial
ring in $r$-variables.

\section{Simple Lie algebras} \label{sec:maintheorem}
In this section, we consider the general case of a
simple Lie algebra $\fg$ over $\C$, fix the notations and state the
main irreducibility theorem, which requires to be modified when the
Lie algebra has at least two roots of different lengths. 

\subsection{Notation}
We first fix the notation and recall relevant facts from the theory of
root systems (see \cite{B},\cite{H}).

 Let $\fg$ be a simple Lie algebra over $\C$, and let
  $\fh$ be a Cartan subalgebra of $\fg$. Denote by
$\Phi\subset \fh^*$ the roots of the pair $(\fg, \fh)$, and let $E$
the real subspace of $\fh^*$ generated by $\Phi$. The dual of the
restriction of the Killing form to $\fh\times \fh$ defines a
non-degenerate symmetric bilinear form on $E$. With respect to
this inner product,
$\Phi$ defines a root system in $E$.

Denote by $\Phi^+\subset \Phi$
(resp. $\Delta\subset \Phi^+$; $\Phi^*\subset
E^*; ~\Phi^{*+}; ~\Delta^*$ )
the subset of positive roots with respect to some ordering of
the root system
(resp. a base for $\Phi^+$;  the set of coroots;  positive
coroots and fundamental coroots). Given a root $\alpha\in \Phi$,
$\alpha^*$ will denote the corresponding coroot.

 Denote by $<.,.>: E^*\times E\to \R$ the duality pairing.
For any root
$\alpha$,  we have $<\alpha^*, \alpha>=2$, and the pairing takes
values in integers when the arguments consist of roots and co-roots.

Let $W$ (resp. $ W^*$) denote the Weyl group of the (resp. dual) root system.
The Weyl group $W$ (resp. $ W^*$) is the subgroup of ${\rm
  Aut}(E)$ (resp. ${\rm Aut}(E^*)$) generated by the
reflections $s_{\alpha}$ of $E$ (resp. $s_{\alpha^*}$ of  $E^*$) defined by
\[s_{\alpha}(u)=u-<\alpha^*, u>\alpha\quad  \text{and}\quad
s_{\alpha^*}(x)=x-<x,\alpha>\alpha^*,\] where $x\in E^*$ and $u\in
E$.  We have $s_{\alpha}(\Phi)\subset \Phi$ and
$s_{\alpha^*}(\Phi^*)\subset \Phi^*$.
 There is a natural isomorphism between the Weyl groups of the root
system and the dual root system, given by $\alpha\mapsto \alpha^*$ and
$s_{\alpha^*}=^t\!s_{\alpha}$ the transpose of $s_{\alpha}$. We
identify the two actions of the Weyl group.

\begin{remark} Sometimes we formulate the propositions for a based
  root system $\cR=(E, \Phi, \Delta)$ instead of $\fg$, and at times
  we refer to the root system just by $\Phi$.
\end{remark}

Let $l(w)$ denote the length of an element in the Weyl group, given by
the least length of a word in the $s_{\alpha}, ~\alpha\in \Delta$
defining $w$. Let $\epsilon(w)=(-1)^{l(w)}$ be the sign character of
$W$.

Denote by $P\subset E$ (resp. $P^+\subset P$, $P^{++}\subset P_+$,
$P^*\subset E^*,~P^{*, ++} $ the lattice of integral weights (resp. dominant
integral weights, dominant regular weights, lattice of integral
co-weights, dominant regular coweights).
Let $r=|\Delta|$ be the rank of $\Phi$. For a
fundamental root $\alpha\in \Delta$, denote by $\omega_{\alpha}$ (resp.
$\omega_{\alpha}^*$) the corresponding fundamental weight
(resp. coweight) defined by
\[ <\beta^*, \omega_{\alpha}>=\delta_{\alpha\beta}\quad  \text{and}
\quad <\omega_{\alpha}^*,\beta>=\delta_{\alpha\beta}, \quad \alpha,
 ~\beta\in \Delta.\]
 The
fundamental weights form a $\Z$-basis for $P$. A weight  $\lambda$ can
 be expressed as a sum,
\[\lambda=\sum_{\alpha\in \Delta}m_{\alpha}(\lambda)\omega_{\alpha},\]
where $m_{\alpha}(\lambda)=<\alpha^*, \lambda>$ are the
coefficients of $\lambda$ with respect to the basis of $P$ determined
by $\Delta$.  The weight $\lambda$ is regular (resp. dominant) if for any
$\alpha\in \Delta$, $m_{\alpha}(\lambda)\neq 0$ (resp. 
 $m_{\alpha}(\lambda)\geq 0$).

\subsection{Schur-Weyl elements}
For any ring $A$ and a commutative group $X$, let $A[X]$ denote the
group algebra of $X$ with coefficients in $A$. We work with
multiplicative (exponential) notation. A basis for $A[X]$ is given
by the elements indexed by
$e^{x}$ for $x\in A$. The group law is expressed by $e^x.e^y=e^{x+y},
~x, ~y\in X$.
The action of the Weyl group is seen as,
$we^{\lambda}=e^{w\lambda}$.

For a weight $\lambda\in P$, define the Schur-Weyl element  $S(\lambda)\in
\Z[P]$ as,
\[ S(\lambda)= \sum_{w\in W}\epsilon(w)e^{w(\lambda)}.\]
The Schur-Weyl elements are alternating with respect to the action of
the Weyl group $W$ on $\Z[P]$,
\[ \sigma(S(\lambda))=\epsilon(\sigma)S(\lambda), \quad \sigma\in W. \]

The group algebra $A[P]$ can be identified with a Laurent polynomial
ring in $r$-variables over $A$, where $r={\rm dim}_{\R}(E)$ is the
rank of $\fg$. In particular, $\C[P]$ is a unique factorization
domain.

\subsection{Weyl Character Formula}

Let $V$ be a finite dimensional $\fg$-module.
  With respect to the
action of $\fh$, there is  a decomposition,
\[ V=\oplus_{\pi\in P} V^{\pi},\]
\[ \text{where} \quad V^{\pi}=\{v\in V\mid Xv=\pi(X)v, ~X\in \fh \}. \]
are the weight spaces of $V$.
The linear forms $\pi$ for which $V^{\pi}$ are non-zero
are the weights of $V$, and $V^{\pi}$ is the subspace of $V$
 consisting of eigenvectors of $H$ with weight $\pi$.   The formal
 character $\chi_V\in \Z[P]$ of $V$ is defined as,
\[ \chi_{{V}}= \sum_{\pi\in P}m(\pi)e^{\pi},\]
where $m(\pi)=\text{dim}(V^{\pi})$ is the multiplicity of $\pi$.
The
character is invariant under the action of the Weyl group.

The irreducible finite dimensional
 $\fg$-modules are indexed by elements in $\lambda\in P^+$, given by
Cartan-Weyl theory. To each dominant, integral weight $\lambda$, we
denote the corresponding irreducible $\fg$-module with highest weight
$\lambda$ by $V_{\lambda}$ and the formal character of $V_{\lambda}$
 by $\chi_{\lambda}$.
The Weyl character formula gives the formula for the
 formal character $\chi_{\lambda}$:
\begin{equation}\label{wcf}
\mbox{\em Weyl character formula:}\quad
\chi_{\lambda}=S({\lambda+\rho})/S({\rho}),
\end{equation}
where $\rho$ is the Weyl weight defined by the equations,
\[ \rho=\frac{1}{2}\sum_{\alpha\in \Phi^+}\alpha=\sum_{\alpha\in
  \Delta}\omega_{\alpha}.\]

\begin{remark}
It can be seen as an application of the Weyl character formula (or
directly by induction), that for $\lambda\in P^{++}$,
the element $S(\lambda)$ is non-vanishing.
\end{remark}

\subsection{Divisibility}
 For any weight
$\lambda$, define
\begin{equation}
d(\lambda)=\mbox{g.c.d.}\{m_{\alpha}(\lambda)\mid \alpha\in \Delta\}.
\end{equation}
Equivalently $d(\lambda)$ can be defined as the largest integer $d$ for
which $d\rho$ divides $\lambda$ (a regular weight $\mu$ is
said to divide a weight $\lambda$ if for every $\alpha\in \Delta$, the
coefficient $m_{\alpha}(\mu)$ divides $m_{\alpha}(\lambda)$).

\begin{example} \label{ex:glrlambda}
 For $SL(r)$, denote by $\omega_i, ~1\leq i\leq r-1$
  the set of fundamental weights given by the highest weights of the
  $i$-th exterior power of the natural representation of $SL(r)$. For
  $\lambda =(a_1, a_2, \cdots, a_r)$ a weight of $GL(r)$, the
  coefficients are given by $m_i(\lambda)=a_i-a_{i+1}$. If the weight
  $\lambda$ is normalized so that $a_r=0$, then the definition of
  $d(\lambda)$ given in the previous section agrees with the above
  definition.
\end{example}

The following proposition, especially the divisibility aspect
 is of fundamental importance to us in this paper, needed in the
 formulation as well as the proof of the main theorem (in establishing
 the inductive argument in the proof of Proposition \ref{prop:key}):
\begin{proposition} \label{prop:weyldenominator}
(a) (Factorization of Weyl denominator).
 For any positive integer $d$, there is a
  factorisation in the Laurent polynomial ring $\Z[\frac{1}{2}P]$,
\begin{align}
 S(d\rho)& =\prod_{\alpha\in \Phi^+}(e^{d\alpha/2}-e^{-d\alpha/2})\\
& = e^{-d\rho}\prod_{\alpha\in \Phi^+}(e^{d\alpha}-1).
\end{align}

(b) (Divisibility by  Weyl denominator).
For any weight $\lambda\in
P$ with greatest common divisor $d(\lambda)$ and any natural number
$d$ dividing $d(\lambda)$  the element
$S(d\rho)$ divides $S(\lambda)$ in $\Z[P]^W$.

(c) (Separability of Weyl denominator) The generalized Weyl
denominators $S(d\rho)$ are separable elements in the ring $\C[P]$
\end{proposition}
\begin{proof}
Part (a) is well known. For $d=1$, (b) is
a restatement of \cite[Proposition 2, Chapter VI, Section 3.3,
  page 185]{B}. For any natural number $d$, there are the  `scaling' maps
 $[d]:\Z[P]\to \Z[P]$ induced by
multiplication by $d$ on $P$. One has,
\[ [d](S(\lambda))=S(d\lambda).\]
Part (b) now follows from the `universal' divisibility
  that $S(\rho)$ divides $S(\lambda)$.

A proof of (c) is given later as Corollary \ref{cor:coprimeroots} in
Section \ref{sec:unique}.
\end{proof}

\subsection{Duality}
It was pointed out by P. Deligne that for a general simple Lie
algebra, there are extra factors which arise whenever there are
at least two roots of different lengths in  a root system for
$\fg$. These extra factorisations arise from duality: the weight
lattices of the root system and its dual are seen to be commensurable
lattices in $E$ as $W$-modules. 
Normalizing the square of the length of the shorter root to be $2$,
this can be seen by the standard identification
$\alpha^*=2\alpha/(\alpha, \alpha)$. If the root system is not simply
laced, then these lattices are not isomorphic. This gives raise to
new factorizations as the Weyl denominator weights $\rho$ and $\rho^*$
are not rational multiples of each other.

Let $\alpha_l$ (resp. $\alpha_s$) be a long (resp. short) root in
$E$. Define,
\begin{equation}
m(\Phi)=(\alpha_l, \alpha_l)/(\alpha_s, \alpha_s).
\end{equation}
The classification of root systems imply that the possible values of
$m(\Phi)$ are,
\[
m(\Phi)=\begin{cases} 1 &\text{if $\Phi$ is of type $A, ~D,~ E$},\\
2 &\text{if $\Phi$ is of type $B, ~C, ~F$},\\
3 &\text{if $\Phi$ is of type $G$}.
\end{cases}
\]
The map
\begin{equation}\label{eqn:dualroot}
\alpha^*\mapsto 2\alpha/(\alpha, \alpha).
\end{equation}
provides a $W$-equivariant identification of the co-weight lattice
$P^*$ with a lattice in $E$ (for the extended action of
$W$ on $E$). Via this identification, if the short root is normalized
to have the square of its length as $2$,  we get
\begin{equation}
\omega_{\alpha}^*=\begin{cases} \omega_{\alpha}  &\text{if $\alpha$ is a short
    root},\\
\omega_{\alpha}/m(\Phi) &\text{if $\alpha$ is a long root}
\end{cases}
\end{equation}
We identify $P^*$ with its image in $E$. The lattices $P$ and $P^*$
are commensurable in $E$. Let
\begin{equation}
\tilde{P}=m(\Phi)P^*.
\end{equation}
Let $\rho^*$ denote the Weyl weight of the dual root system defined
by $\Phi^*$. Define
\begin{equation}
\tilde{\rho}= m(\Phi)\rho^*=\sum_{\alpha\in \Delta_s}
m(\Phi)\omega_{\alpha}+\sum_{\alpha\in \Delta_l}\omega_{\alpha},
\end{equation}
where $\Delta_s=\Delta\cap \Phi_s$ (resp. $\Delta_l=\Delta\cap
\Phi_l$), and $\Phi_s$ (resp. $\Phi_l$)  is the subset of $\Phi$
consisting of the short (resp. long) roots. From the formula for $m(\Phi)$,
it follows that $\tilde{\rho}=\rho$ precisely for the simply laced
root systems (of type $A, ~D, ~E$). It can be seen that
the element $\tilde{\rho}$ is a generator of the group
$\Q \rho^*\cap P$.

The Weyl group invariance
of the pairing allows us to transfer the factorization and
divisibility relations for the Schur-Weyl sums of  integral multiples
of $\rho^*$  in the group ring $\Z[P^*]$
to the group ring $\Z[P]$ of the lattice $P$. The point is that
when there are at least two roots having different lengths, this
gives us new factorizations and divisibility relations.
This is expressed by the following proposition `dual' to Proposition
\ref{prop:weyldenominator}:

\begin{proposition} \label{prop:dualrhodiv}

(a) (Factorization of dual Weyl denominator).
 For any positive integer $d$, there is a
  factorisation in the Laurent polynomial ring $\Z[\frac{1}{2}P]$,
\begin{align}
 S(d\tilde{\rho})& =\prod_{\alpha\in
   \Phi_s^+}(e^{dm\alpha/2}-e^{-dm\alpha/2})
\prod_{\alpha\in \Phi_l^+}(e^{d\alpha/2}-e^{-d\alpha/2})\\
& = e^{-d\tilde{\rho}}\prod_{\alpha\in \Phi_s^+}
(e^{dm\alpha}-1)\prod_{\alpha\in \Phi_l^+}(e^{d\alpha}-1).
\end{align}

(b) (Divisibility by dual Weyl denominator).
If $d\tilde{\rho}|\lambda$, then  $S(d\tilde{\rho})$ divides $S(\lambda)$ in
$\Z[P]$.

(c) (Separability of dual Weyl denominator) The elements $
S(d\tilde{\rho})$ are separable in the ring $\C[P]$.
\end{proposition}

We refer to elements of the form $S(d\rho)$ or $S(d\tilde{\rho})$ as
elements of {\em (generalized) Weyl denominator type} in $\Z[P]$.

\subsection{Factors of $S(\lambda)$ of Weyl denominator type}
For $\lambda\neq \rho$ a dominant regular weight in $P^{++}$,
define
\begin{equation}\label{def:dlambda}
 D(\lambda)={\rm l.c.m.}_{\substack{d\rho|\lambda, ~d\rho\neq
    \lambda\\
e\tilde{\rho}|\lambda, ~e\tilde{\rho}\neq\lambda}}
    (S(d\rho),S(e\tilde{\rho})),
\end{equation}
where the l.c.m. is taken in the Laurent polynomial ring $\Z[P]^W$.
By \cite[Theorem 1, Ch. VI, Section 4, page 188]{B}, the ring
$\Z[P]^W$ is isomorphic to a polynomial ring over $\Z$ in
$r$-variables:
\[ \Z[P]^W\simeq \Z[\{\omega_{\alpha}\mid \alpha\in \Delta\}].\]
Hence the units of the ring $\Z[P]^W$ are isomorphic to $\{1, -1\}$.
With respect to the dominant order (a weight is non-negative if
it can be written as a non-negative linear combination of positive roots),
 the leading term of $S(\lambda)$
is given by $e^{\lambda}$.
The least common multiple in the definition of $D(\lambda)$ in the
above equation is taken to be the element whose coefficient of the leading
monomial occuring in $D(\lambda)$ with respect to the dominant
ordering is positive (equal to $1$). Define,
\begin{equation} \label{def:clambda}
 C(\lambda)=S(\lambda)/D(\lambda),
\end{equation}
to be the quotient of the Schur-Weyl sum $S(\lambda)$ divided by the
obvious factors arising from the Weyl character formula and duality.
Since $S(\lambda)$ is alternating, it follows that $C(\lambda)\in
\Z[P]^W$.

We now look at the structure of $D(\lambda)$.
If $\lambda=m(\Phi)\lambda^*$ for some
$\lambda^*\in P^*$,  define
\begin{equation}\label{eqn:d*}
 d^*(\lambda)=d(\lambda^*)={\rm g.c.d}\{m_{\alpha}^*(\lambda^*)\mid
\alpha\in\Delta\}.
\end{equation}
Equivalently $d^*(\lambda)$ can be defined as the largest integer $d$ for
which $d\tilde{\rho}$ divides $\lambda$.
Note that we have the following inclusions:
\begin{equation}\label{eqn:latticeinc}
 P^*\supset P\supset m(\Phi)P^*\supset m(\Phi)P.
\end{equation}

\begin{lemma}\label{lem:Dlambda}
Suppose that $\lambda\in P^{++}$ is neither  a multiple of  $\rho$ or
$\tilde{\rho}$. Then
\begin{equation}\label{Dlambda2}
D(\lambda)=\begin{cases} S(d(\lambda)\rho) &\text{if
  $\lambda\in m(\Phi)^iP\backslash m(\Phi)^{i+1}P^*$ for some $i\geq 0$},\\
S(d^*(\lambda)\tilde{\rho}),&\text{if $\lambda\in
  m(\Phi)^{i+1}P^*\backslash m(\Phi)^{i+1}P$ for some $i\geq 0$}.
\end{cases}
\end{equation}
\end{lemma}
\begin{proof}
An element $\lambda \in  m(\Phi)^iP\backslash m(\Phi)^{i+1}P^*$ for
some $i\geq 0$ iff
$m(\Phi)^i|m_{\alpha}(\lambda), ~\forall \alpha \in \Delta$ and
there exists a $\alpha\in \Delta_s$ such that $m(\Phi)^{i+1}$ does not
divide $m_{\alpha}(\lambda)$. If $i=0$, then
  $\tilde{\rho}$ does not divide $\lambda$.
For $i>0$, suppose $em(\Phi)^j\tilde{\rho}$
  divides $\lambda$, with $e$ coprime to $m(\Phi)$. This implies that
 $em(\Phi)^{j+1}|m_{\alpha}(\lambda)$ for all $\alpha\in
 \Delta_s$. Hence,  $j\leq i-1$, and it follows that
 $em(\Phi)^{j+1}\rho|\lambda$. But
  $em(\Phi)^j\tilde{\rho}$ divides  $em(\Phi)^{j+1}{\rho}$. Thus the
lcm can be taken amongst factors of the form $S(f\rho)$ with
$f\rho|\lambda$. Since $\lambda$ is not a multiple of $\rho$, the lcm
is given by $S(d(\lambda)\rho)$, and this proves the first case.

In the second case, if $em(\Phi)^j{\rho}$ divides $\lambda$ with $e$
coprime to $m(\Phi)$, then $j\leq i$. Since $\lambda\in
m(\Phi)^{i+1}P^*$, it follows that $\lambda$ is divisible by
$em(\Phi)^j\tilde{\rho}$, and hence the lcm can be taken with respect
to such factors. This establishes the second case.

\end{proof}

\subsection{The Main theorem}
Our aim is to show that if $\lambda\in P^{++}$ is neither a multiple
of $\rho$ or of $\tilde{\rho}$ then $C(\lambda)$ is absolutely
irreducible. However the proof we have does not prove this in full
generality and has a gap for a class of regular weights of $G_2$ and
$F_4$. We make the following assumption (see Proposition \ref{prop:nonmonic}):

{\em Assumption NMFG:} Consider the root systems given by $F_4$ and
$G_2$, and regular weights  $\lambda\in P^{++}$ be of the form,
\[ \lambda=u\omega_{\alpha}+v\omega_{\beta}+d(\lambda)\rho,
\quad d(\lambda)=(u,v)\]
where $\omega_{\alpha}$ (resp. $\omega_{\beta}$) is the fundamental
weight corresponding to the short (resp. long) corner root $\alpha$
(resp. $\beta$) in the Dynkin diagram such that the following
inequalities are satsified:
\[ m(\Phi)v \geq u+d(\lambda)\quad \text{and}\quad u\geq v+d(\lambda).\]
{\em Assumption NMFG} is that for this class of weights, any
$W$-invariant non-trivial factorization is non-monic
(see Definition
\ref{dfn:nonmonic} for the definition of a factorization to be non-monic).

The main theorem of this
paper is the following:
\begin{theorem}\label{thm:main}
With notation as above,
let $\lambda\in P^{++}$ be a dominant
regular weight for the root system $\Phi$. Assume further that
Assumption NMFG is valid.
If $\lambda\neq d\rho$ or $d\tilde{\rho}$ for some natural number
  $d$, then $C(\lambda)$ is absolutely irreducible, i.e.,
 it  is irreducible in the ring
$\C[P]^W$.
\end{theorem}

\begin{remark}
As for the case of $GL(r)$, it should be possible to obtain the
irreducibility statement in the larger ring $\C[P]$, but we do not
carry out this reduction out here. Analogous irreducibility results
can be obtained for
characters of rational representations of simple algebraic groups $G$.
In fact, it is possible to obtain such irreducibility results in the
bigger ring of regular  functions of $G$, rather than the ring of
regular invariant functions of $G$ (see Theorem \ref{thm:irrgroups}).
\end{remark}

\begin{remark}\label{rem:duality-red-pf}
Lemma \ref{lem:Dlambda} allows us by the use of duality to reduce the
proof of Theorem \ref{thm:main} to the case that
\[\lambda\in m(\Phi)^iP\backslash m(\Phi)^{i+1}P^*,\]
for some $i\geq 0$. In this case, $D(\lambda)=S(d(\lambda)\rho)$ and
we need to show that $C(\lambda)=S(\lambda)/S(d(\lambda)\rho)$ is
irreducible. We will achieve this by showing that if $C(\lambda)$ is
reducible, then $\lambda\in  m(\Phi)^{i+1}P^*$, contradicting our
choice of $\lambda$ (see Section \ref{sec:keyprops}).
\end{remark}

\begin{remark} The scaling operation $[d]:\Z[P]\to \Z[P]$
given by mulitplication by $d$ on $P$ allows a
reformulation of  the theorem using
a slight modification of a ring used by
Bourbaki \cite{B} in their formulation
of the product expansion of the Weyl denominator. Let
\[ P_{\Q}= P\otimes_{\Z}\Q,\]
and denote by $\C[P_{\Q}]$ denote the group algebra of $P_{\Q}$ with
coefficients in $\C$ (we use exponential notation).
Observe that in this ring there exist elements which are infinitely
factorizable, for example elements of the form $e^p-1$ for $p\in P$.
 Further an element $C \in \C[P]$ is irreducible as an
element in $\C[P_{\Q}]$ if and only if
$[d]C$ is irreducible in $\C[P]$ for all natural
numbers $d$, where $[d]:\C[P]\to \C[P]$ denotes the algebra
homomorphism induced by multiplication by $d$.
  Since $P$ and the dual weight lattice $P^*$ are
commensurable in $P_{\Q}$, the two group rings $\C[P]$ and $\C[P^*]$
are both contained in $\C[P_{\Q}]$. In this ring of fractional Laurent
polynomials,
Theorem \ref{thm:main} can be reformulated as saying that with the
hypothesis of Theorem \ref{thm:main},
 the element
$C(\lambda)$ is irreducible in the ring $\C[P_{\Q}]^W$, where we can
assume that $d(\lambda)=1$.

However we do not work with this ring any further, as there is no
clear advantage in working with this bigger ring.
\end{remark}

\subsection{An uniquness property}
The following theorem expresses an uniqueness property of
`generalized characters'; in particular, that
the highest weight $\lambda$ can be recovered from knowing
$C(\lambda)$ provided $C(\lambda)$ is non-trivial:
\begin{theorem}\label{thm:unique}
Let $\cR=(E,\Phi, \Delta)$ be  a simple based root system
and $\lambda_i,  ~\mu_i, ~i=1,2$
 be dominant regular weights for
$\cR$. Assume that  the weights $\mu_1, ~\mu_2$ are of generalized
Weyl denominator type, i.e., they are an  integral multiple of either
$\rho$ or $\tilde{\rho}$. Assume further that $\mu_i$ divides $\lambda_i$ for
$i=1, 2$.   Suppose
 there is an equality of the quotients,
\[S(\lambda_1)/S(\mu_1)=S(\lambda_2)/S(\mu_2),  \]
and that these quotients are not equal to $1$.

Then $\lambda_1=\lambda_2$ and $\mu_1=\mu_2$.
In particular, if $\lambda_1$ is neither a multiple of $\rho$
nor $\tilde{\rho}$, and
$C(\lambda_1)=C(\lambda_2)$, then $\lambda_1=\lambda_2$.

\end{theorem}

The proof of the theorem is given in Section \ref{sec:unique}.
Together with Theorem \ref{thm:main} and the explicit factorizations
of the generalized Schur-Weyl denominators, this gives a proof of
 the unique factorization of tensor products given by
Theorem \ref{thm:tensor}, but subject to Hypothesis {\em NMFG}. We
refer to Section \ref{sec:unique} for more details.

\begin{remark} The uniqueness property is required for the proof of
Proposition \ref{prop:nonmonic} that any symmetric factorization is
non-monic, which in turn goes into the proof
  of Theorem \ref{thm:main}.
\end{remark}

\section{Irreducibility property for  characters  of irreducible
representations of  simple  algebraic groups} \label{sec:groups}
In this section we
extend the irreducibility results of the previous section  to
characters of finite dimensional representations of  simple algebraic
groups.  For an algebraic group $H$, denote by ${\mathcal O}(H)$ the
algebra of regular functions on $H$. If a  group $L$ acts on $H$, we
denote by ${\mathcal O}(H)^L$ the ring of regular functions on $H$
which are  invariant with respect to the induced action of  $L$ on
${\mathcal O}(H)$.

Let $G$ be a connected,  simply
connected, almost simple algebraic group over $\C$.  Since $G$ is
simply connected, by a theorem of Fossum, Iversen and Popov  (see
\cite{FI}, \cite{KKLV1}, \cite{KKLV2}),  the Picard group of $G$ is
trivial. Hence  the ring ${\mathcal O}(G)$ is factorial. We have the
following:

\begin{proposition} \label{invirr} Let $G$ be a connected,  simply
connected, almost simple algebraic group over $\C$.  Suppose $f\in
{\mathcal O}(G)^G$ is an invariant regular function on $G$ with
respect to the  adjoint action of $G$ on itself. Then any irreducible
factor of $f$ in ${\mathcal O}(G)$ is  invariant with respect to the
adjoint action of $G$ on itself.

In particular, if $f$ is an irreducible element in ${\mathcal
O}(G)^G$, then it is irreducible in ${\mathcal O}(G)$.
\end{proposition}
\begin{proof} Suppose there is a factorization
\begin{equation}\label{factor} f=p_1\cdots p_r,
\end{equation} in the ring ${\mathcal O}(G)$, where $p_i$ are
irreducible elements in ${\mathcal O}(G)$. The group $G(\C)$ acts by
conjugation on ${\mathcal O}(G)$ leaving invariant the element $f$;
hence it acts by permuting the irreducible factors (upto units) $p_1,
\cdots, p_r$. Since $G(\C)$ is connected, this implies that the
permutation action is trivial,
\[ p_i^g=\xi_i(g)p_i, \quad i=1, \cdots, r, \] where $\xi_i(g)$ is a
nowhere vanishing function on $G(\C)$, and satisfies the $1$-cocycle
condition,
\[ \xi_i(gh)= \xi_i(g)^h \xi_i(h).\] From the regularity of the action
of $G$ on ${\mathcal O}(G)$, we conclude that $\xi_i(g)$ is a regular
function on $G$, hence an unit in ${\mathcal O}(G)$.  By a theorem of
Rosenlicht (\cite[page 78]{KKLV1}) the units in ${\mathcal O}[G]$ are
just the constants. Thus the $G$ action is trivial on the units of
${\mathcal O}[G]$.  This defines a homomorphism  $g\mapsto\xi_g$ from
the group $G({\mathcal O})$ to $\C^*$. Since $G(\C)$ has no abelian
quotients, this implies the cocyle is trivial, and hence $p_i^g=p_i$
for any $g\in G(\C)$.  Hence the factorization given by Equation
\ref{factor} actually holds in ${\mathcal O}(G)^G$.

\end{proof} Let $T$ be a maximal torus in $G$. By Chevalley's
restriction theorem, we have an isomorphism
\[ {\mathcal O}(G)^G\simeq {\mathcal O}(T)^W,\]  between the algebra
${\mathcal O}[G]^G$ and the algebra ${\mathcal O}(T)^W$ of Weyl group
invariant functions of $H$. Let $X^*(T)$ denote the group of
characters of $T$. The ring of regular functions ${\mathcal O}(T)$ on
the torus $T$ can be identified with  the group algebra
$\C[X^*(T)]$. Since $G$ is simply connected, by a theorem of Chevalley
it is known that these rings are isomorphic to the polynomial ring in
$r$-variables, where $r$ is the  dimension of $T$.  The Lie algebra
$\fg$ of $G$ is simple.  Choosing a Borel subgroup $B\supset T$ of $G$
allows us to define simple roots, weights, etc. for $\fg$ too.  The
lattice of weights $P$ can be identified with the character group
$X^*(T)$ of $T$.  Hence we have an identification,
\begin{equation}\label{eqn:group-liealg} {\mathcal O}(T)^W\simeq
\C[P]^W.
\end{equation}

To each dominant, integral weight $\lambda$,  denote the corresponding
irreducible $G$-module with highest weight $\lambda$ by
$V_{\lambda}$. Via the above isomorphism given by Equation
\ref{eqn:group-liealg},  the characters of the representation of $G$
and the Lie algebra on $V_{\lambda}$ can be identified.  In
particular, the irreducibility results of the previous section can be
transferred to the  context of invariant functions on the group. But
we obtain a bit more by  combining Theorem \ref{thm:main} and
Proposition \ref{invirr} (we have also incorporated the scaling operation):

\begin{theorem}\label{thm:irrgroups} Let $G$ be a  connected, simply
  connected, almost
 simple algebraic group over $\C$ of rank at least two. With respect
 to notation as above,
let  $\lambda$  be the highest weight of an
irreducible representation of $G$.
Suppose that $\lambda+\rho$ is not  a multiple of either $\rho$ or
$\tilde{\rho}$, and that $d(\lambda+\rho)=1$. Assume further
that Assumption NMFG holds for the weight $\lambda +\rho$.

Then   for any natural
number $d$, the function $g\mapsto \chi_{\lambda}(g^d)$ is  irreducible in the
ring of regular functions of $G$.
\end{theorem}

\begin{remark} For $G\simeq SL(r), ~r\geq 3$, this is the theorem
  mentioned in the abstract.
\end{remark}

\section{Cofactor expansions}  \label{sec:cofactor}
The proof of the unique decomposition
of tensor products of  irreducible representations of simple Lie
algebras (see Theorem \ref{thm:tensor}) given in \cite{R}
is by induction on the rank of the Lie algebra, by considering
cofactor expansions of the Schur-Weyl elements occuring in the Weyl
character formula. The same general principle is applied to  the proof
of the irreducibility property of characters, with the expectation
that an inductive machinery can be setup.

\subsection{Cofactor expansion for $GL(r)$}
We first recall the cofactor expansion of the numerator of the Weyl
character formula for $GL(r)$.
Let $\lambda =(a_1>a_2> \cdots >a_{r-1}>  a_r=0)$ be a regular weight of
$GL(r)$.  The
Schur-Weyl sum $S(\lambda)$ can be expressed as a determinant,
\begin{equation*}
 \begin{split}
S(\lambda)& =\sum_{\sigma\in W}\epsilon(\sigma)
x_{\sigma(1)}^{a_1}\cdots x_{\sigma(r)}^{a_r}\\
& ={\rm det}(x_i^{a_j}).
\end{split}
\end{equation*}
This admits a cofactor expansion,
\begin{equation}\label{eqn:cofactorexpglr}
  \begin{split}
S(\lambda)(x_1,\cdots, x_r) & = x_1^{a_1}
S(\lambda^{(1)})(x_2,\cdots, x_r)-
x_1^{a_2}S(\lambda^{(2)})(x_2,\cdots, x_r)+ \\
& \cdots+
(-1)^{r-2}x_1^{a_{r-2}}
S(\lambda^{(r-2)})(x_2,\cdots, x_r)\\
& +(-1)^{r-1}
(x_2\cdots x_r)^{a_{r-1}}S(\lambda^{(r-1)})(x_2,\cdots, x_r),
\end{split}
\end{equation}
where for $1\leq i \leq r-2$
\[
 \lambda^{(i)}  =(a_1, a_2, \cdots, a_{i-1}, a_{i+1},  \cdots, a_{r}),\]
and
\[
 \lambda^{(r-1)} =(a_1-a_{r-1}, \cdots, a_{r-2}-a_{r-1}, 0),\]
are regular  weights for $GL(r-1)$. For our purpose, we are interested
only in the top (resp. bottom) two leading terms, and not the
full cofactor expansion as such.

In terms of the fundamental weights defined as in Example
\ref{ex:glrlambda}, if $\lambda=\sum_{i=1}^{r-1}m_i(\lambda)\omega_i$
with $m_i(\lambda)=a_i-a_{i+1}$, these weights can be expressed as,
\begin{align}\label{eqn:glrcofactor}
\lambda^{(1)} & =m_2(\lambda)\omega_1^{r-1}+ \cdots +m_{r-1}(\lambda)\omega_{r-2}^{r-1},\\
\lambda^{(2)} & =(m_1(\lambda)+m_2(\lambda))\omega_1^{r-1}+ \cdots
+m_{r-1}(\lambda)\omega_{r-2}^{r-1},\\
\lambda^{(r-1)} & =m_1(\lambda)\omega_1^{r-1}+ \cdots
+m_{r-2}(\lambda)\omega_{r-2}^{r-1}
\end{align}
where we have put a superscript $r-1$ to indicate that these are
fundamental weights for $GL(r-1)$.

\subsection{Cofactor expansion}
Our aim is to generalize the foregoing cofactor expansion for $GL(r)$
to that of a general simple based root system ${\mathcal R}=(E, \Phi,
\Delta)$ of rank $r$. The above interpretation of the cofactor
expansion in terms of the fundamental weights leads us to consider
corner roots in the Dynkin diagram of $\cR$ and to decompose
the weight lattice $P$ as a sum of the weight lattice corresponding to
the simple Lie subalgebra corresponding to the corner root of corank
one and the fundamental weight given by the corner root.

Choose a fundamental root
$\alpha\in \Delta$. We will be primarily
interested in the  special case  when
$\alpha$ corresponds to a corner vertex in the Dynkin diagram of
${\mathcal R}$.  Let $\Delta_{\alpha}=\Delta\backslash \{\alpha\}$,
and let $\Phi_{\alpha}\subset \Phi$ be the subset of roots lying in the span
of the roots generated by $\Delta_{\alpha}$. Let
\[ E_{\alpha}= \sum_{\alpha\in \Delta_{\alpha}}\R\alpha=\{\mu\in E\mid
<\omega_{\alpha}^*, \mu>=0\}. \]
It is known that $\cR_{\alpha}=(E_{\alpha},\Phi_{\alpha}, \Delta_{\alpha})$ is a
 based simple root system
 of rank $r-1$.

Let $W_{\alpha}$ denote the Weyl group of
$\cR_{\alpha}$. It can be
identified with the subgroup of $W$ generated by the fundamental
reflections $s_{\beta}$ for $\beta\in \Delta_{\alpha}$.
The following lemma provides a  $W_{\alpha}$-equivariant decomposition
of $P$, a  complement to $E_{\alpha}$
inside $E$ (see \cite[Lemma 3]{R}):
\begin{lemma} \label{lem:decomp}
Let $\omega_{\alpha}$ (resp. $\omega_{\alpha}^*$) denote
  the fundamental weight (resp. coweight) corresponding to the fundamental root
  $\alpha$. The isostropy group of $\omega_{\alpha}^*$ is precisely
  $W_{\alpha}$. There is $W_{\alpha}$-equivariant decomposition,
\[E=\R\omega_{\alpha}\oplus E_{\alpha}\quad \mbox{and}\quad
E^*=\R\omega_{\alpha}^*\oplus{E_{\alpha}}^*. \]
Via the above decomposition, the  rational weight lattice
$P\otimes \Q$ admits a $W_{\alpha}$-equivariant splitting,
\[ P\otimes \Q= \Q\omega_{\alpha} \oplus P_{\alpha}\otimes \Q,\]
where  the weight lattice $P_{\alpha}$ of the
root system  $\Phi_{\alpha}$ can be identified with the subspace of
$P$,
\begin{equation}\label{eqn:Palpha}
 P_{\alpha}= {\rm Ker}(\omega_{\alpha}^*)= \{ \pi\in P\mid
<\omega_{\alpha}^*,\pi>=0\}.
\end{equation}
Further, there is a $W_{\alpha}$-equivariant inclusion,
\begin{equation}\label{eqn:cof-dec-P}
 P\subset \Z\frac{\omega_{\alpha}}{<\omega_{\alpha}^*,
  \omega_{\alpha}>}\oplus P_{\alpha}.
\end{equation}
\end{lemma}
The last assertion follows from the fact if $\mu\in P$ takes integral
values on the fundamental coroots, then  it's projection to
$E_{\alpha}$ also takes integral values on the fundamental coroots
$\beta^*, ~\beta\in \Delta_{\alpha}$, as $\omega_{\alpha}$ is
orthogonal to all such $\beta^*$.

Denote by  $l_{\alpha}$ the rational weight,
\begin{equation}\label{eqn:lalpha}
l_{\alpha}=\frac{1}{<\omega_{\alpha}^*,
  \omega_{\alpha}>}\omega_{\alpha}.
\end{equation}
With respect to the decomposition, a weight $\pi\in P_0$ can be
written as,
\begin{equation}\label{eqn:cof-dec-weight}
 \pi=\omega_{\alpha}^*(\pi)\omega_{\alpha} +\pi^{\alpha},
\end{equation}
where $\pi^{\alpha}\in P_{\alpha, 0}= P_{\alpha}\otimes \Q$ d
efined by the above equation,  is the
$W_{\alpha}$-equivariant projection of $\pi$ along $\omega_{\alpha}$
to $E_{\alpha}$. The integer $\omega_{\alpha}^*(\pi)$ (or
the rational number $\omega_{\alpha}^*(\pi)/\omega_{\alpha}^*(\omega_{\alpha})$
will be referred to as the degree of $\pi$ along $l_{\alpha}$
(or along $\omega_{\alpha}$ or $\alpha$).

 Let $\omega$ be a fundamental weight of $\Phi$ distinct from
$\omega_{\alpha}$. It follows from Equation \ref{eqn:cof-dec-weight}
that $\omega^{\alpha}$ is a fundamental weight for the root system
$\Phi_{\alpha}$.  In particular, the projection  $\rho^{\alpha}$ of
$\rho=\rho(\Phi)$ is the sum of the fundamental weights of
$\Phi_{\alpha}$. i.e., equal to the Weyl weight of the root system
$\Phi_{\alpha}$.

Suppose $U$ is an element of $\C[P]$. Write
\[ U=\sum_{\mu\in P}a_{\mu}(U)e^{\mu}. \]
Define $P(U)$ to be the set finite of weights occuring in $U$,
\[ P(U)=\{\mu\in P_0\mid a_{\mu}(U)\neq 0\}, \]
Expand $U$ in terms of the `degree along $\alpha$' as,
\[
 U=\sum_{i\geq 0}U'_{\alpha,u-i},
\]
where
\[U'_{\alpha, u-i}=\sum_{\mu\in P(U),~ \omega_{\alpha}^*(\mu)=u-i}a_{\mu}(U)e^{\mu}.\]
If $U$ is $W$-invariant, then the terms $U'_{\alpha,u-i}$ are
$W_{\alpha}$-invariant. The term $U'_{\alpha,u-i}$ can also be written
as,
\[U'_{\alpha,u-i}=e^{(u-i)l_{\alpha}}U_{\alpha, u-i}, \]
where we now consider $U_{\alpha,i}$ as an element of
$\C[P_{\alpha}]$. Define the {\em cofactor   expansion} of $U$ along
$\alpha$ as,
\begin{equation}\label{eqn:cofexp}
 U=\sum_{i\geq 0}e^{(u-i)l_{\alpha}}U_{\alpha, u-i}.
\end{equation}
The element $U_{\alpha, u}\in \C[P_{\alpha}]$ will also be referred
to as the {\em leading coefficient of $U$ along $\alpha$}.

For any
integer $i\geq 0$, we refer to the term $ U_{\alpha, u-i}$ as the {\em
 $i$-th codegree term} in the cofactor expansion of $U$ along $\alpha$
(see Section \ref{sec:unique} where we use this notation).

\begin{example}\label{ex:glrleadingcoeff}
If we consider the cofactor expansion of $S(\lambda)$
for $GL(r)$, the top degree coefficient is the leading
coefficient corresponding to the cofactor expansion along the
fundamental corner root $e_1-e_2$ (standard notation). Upto
a monomial term, the constant term is the leading coefficient in the
cofactor expansion of $S(\lambda)$ along the other corner fundamental
root $e_{r-1}-e_r$.
\end{example}

\subsection{Cofactor expansion of $S(\lambda)$}
We now describe the cofactor expansion of the Schur-Weyl  sum $S(\lambda)$.
Let $W(\alpha)$ be a set of right coset representatives for $W_{\alpha}$ in
$W$, i.e., a section for the projection map $W\to W_{\alpha}\backslash W$. For
any element $w\in W_{\alpha}$ and $s\in W$,
the value
\[
\omega_{\alpha}^*(ws\lambda)=(w^{-1}\omega_{\alpha}^*)(s\lambda)=
\omega_{\alpha}^*(s\lambda)\]
is a constant for any weight $\lambda$. Hence the Schur-Weyl sum
$S(\lambda)$ can be expanded as,
\begin{equation}\label{eqn:cofexp}
\begin{split}
 S(\lambda)& =
\sum_{d\in \Z}e^{dl_{\alpha}}\left(\sum_{w\in
W_d}\epsilon(w)e^{(w\lambda)^{\alpha}}\right),\\
&=\sum_{s\in W(\alpha)}e^{\omega_{\alpha}^*(s\lambda)l_{\alpha}}
S((s\lambda)^{\alpha})\\
\text{where}\quad  W_d&
=\{w\in W\mid \omega_{\alpha}^*(w\lambda)=d\},
\end{split}
\end{equation}
and for each $s\in W(\alpha)$, $S((s\lambda)^{\alpha})$ refers to the
Schur-Weyl sum of the weight $(s\lambda)^{\alpha}$ belonging to the root system
$\cR_{\alpha}$. In the above notation,
\[ S(\lambda)_{\alpha,d}=\sum_{s}S((s\lambda)^{\alpha}),\]
where the sum ranges over
$s\in W(\alpha)$ such that $ \omega_{\alpha}^*(s\lambda)=d$.

 We are interested in the first two leading terms in the above
expansion. Given a regular weight $\lambda\in P_{+}$, define
\[
\begin{split}
a_{\alpha, 1}(\lambda) & =
\text{max}\{w\lambda(\omega_{\alpha}^*)\mid w\in W\}, \\
 a_{\alpha, 2}(\lambda)&
= \text{max}\{w\lambda(\omega_{\alpha}^*)\mid w\in W~ \text{and} ~
w\lambda(\omega_{\alpha}^*)\neq a_{\alpha,1}(\lambda)\}.
\end{split}
\]
The following lemma is proved in \cite[Lemma 4]{R}:
\begin{lemma} \label{formalism}
Let $\lambda$ be a regular weight in $P_+$ and $\alpha\in \Delta$.
\begin{enumerate}
\item The largest value
$a_{\alpha,1}(\lambda)$ of $(w\lambda)(\omega_{\alpha}^*)$ for $w\in
  W$,  is attained
precisely for
$w\in W_{\alpha}$. In particular,
\[ a_{\alpha,1}(\lambda)=<\omega_{\alpha}^*,\lambda>.\]
\item The second highest value of $a_{\alpha,2}(\lambda)$ is attained
precisely for $w$ in $W_{\alpha}s_{\alpha}$, and the value is given by
\[a_{\alpha,2}(\lambda)=\omega_{\alpha}^*(s_{\alpha}\lambda)=
a_{\alpha,1}(\lambda)-<\alpha^*,
\lambda>=a_{\alpha,1}(\lambda)-m_{\alpha}(\lambda).\]
\end{enumerate}
\end{lemma}

 Assume now that $\alpha$ is a corner root. Then $\cR_{\alpha}$ is
 simple. As a corollary of the above discussion, we obtain the first two terms
for the cofactor
expansion of $S(\lambda)$
 along $\omega_{\alpha}$:
\begin{lemma}\label{lem:cofexp}
With notation as above, let
$\lambda=\sum_{\beta\in \Delta}m_{\beta}(\lambda)\omega_{\beta}$.
The cofactor  expansion of $S(\lambda)$
 given by Equation (\ref{eqn:cofexp}) is,
\begin{equation}
S(\lambda) =e^{a_{\alpha,1}(\lambda)l_{\alpha}}S(\lambda^{\alpha})
- e^{a_{\alpha,2}(\lambda)l_{\alpha}}S((s_{\alpha}\lambda)^{\alpha})
+L(\lambda),
\end{equation}
where $L(\lambda)$ denotes the terms of degree along $l_{\alpha}$ less than
the second highest degree. In terms of fundamental weights,
\begin{align}\label{eqn:salpha}
\lambda^{\alpha}&=\sum_{\beta \in \Delta_{\alpha}}
m_{\beta}(\lambda)\omega_{\beta}^{\alpha},\\
(s_{\alpha}\lambda)^{\alpha}& =\lambda^{\alpha}+|m_{\alpha_n}(\alpha)|m_{\alpha}(\lambda)\omega_{\alpha_n}^{\alpha}
\end{align}
where  $\alpha_n$ is the unique
root connected to $\alpha$ in the Dynkin diagram of
$\cR$.
\end{lemma}
\begin{proof}
We need to prove only the last formula. By definition,
\[ s_{\alpha}\lambda=\lambda-<\alpha^*,
\lambda>\alpha=\lambda-m_{\alpha}(\lambda)\alpha.\]
Since $\cR$ is simple and $\alpha$ is a corner root, the term
$m_{\beta}(\alpha)=<\beta^*, \alpha>$ vanishes if $\beta$ is different
from $\alpha$ and $\alpha_n$. Hence, in terms of fundamental weights,
\begin{equation} \alpha=\sum_{\beta\in \Delta}m_{\beta}(\alpha)\omega_{\beta}=
2\omega_{\alpha}+m_{\alpha_n}(\alpha)\omega_{\alpha_n},
\end{equation}
where $m_{\alpha_n}(\alpha)= <\alpha_n^*,\alpha>$ is a negative
integer.  Putting all this together yields
\begin{equation}\label{eqn:salphalambda}
\begin{split}
(s_{\alpha}\lambda)^{\alpha}&=\lambda^{\alpha}+|m_{\alpha_n}(\alpha)|m_{\alpha}(\lambda)\omega_{\alpha_n}^{\alpha}\\
&=
(m_{\alpha_n}(\lambda)+|m_{\alpha_n}(\alpha)|m_{\alpha}(\lambda))\omega_{\alpha_n}^{\alpha}+\sum_{\beta\neq
  \alpha, ~\alpha_n}m_{\beta}(\lambda)\omega_{\beta}^{\alpha}.
\end{split}
\end{equation}

\end{proof}

\begin{example} For $GL(r), ~|m_{\alpha_n}(\alpha)|=1$, and this gives
  the formula expressed in Equation \ref{eqn:glrcofactor}:
\[ \lambda^{(2)}=m_1(\lambda)\omega_1^{(r-1)} +\lambda^{(1)}.\]
\end{example}

\subsection{Eisenstein criterion}
We now state the equivalent in our context of the observation used in
the proof of the classical Eisenstein criterion regarding irreducibility of
polynomials. We first assert the separbility of $S(\lambda)$:
\begin{lemma} Assume that Theorem \ref{thm:main} holds for the based
  simple root
  system $\cR=(E, \Phi,\Delta)$. Then for any dominant regular weight $\mu$,
  the Schur-Weyl sum $S(\mu)$ is a separable element in $\C[P]$.
\end{lemma}
\begin{proof} This is a simple consequence of the hypothesis and the
  separability of generalized Weyl denominators given in Part (c) of
  Propositions \ref{prop:weyldenominator} and \ref{prop:dualrhodiv}
  (see also Section \ref{sec:unique}).
\end{proof}

We now formulate  the classical Eisenstein criterion in
our context:
\begin{lemma}[Eisenstein criterion]\label{lem:eisencrit} Assume that
Theorem \ref{thm:main} holds for any irreducible based root system of
rank less than that of $\cR=(E, \Phi,\Delta)$.  Let $\lambda$ be a dominant
regular weight  and
$U$ be a factor of $S(\lambda)$. Then  for any $\alpha\in \Delta$ and
for any $i< m_{\alpha}(\lambda)$,  the leading term $U_{\alpha,u} $
 in the cofactor expansion given by
Equation \ref{eqn:cofexp} divides $U_{\alpha,u-i}$ in the ring $\C[P_{\alpha}]$.
\end{lemma}
\begin{proof} Let $\mu=\lambda^{\alpha}$. By the previous lemma,
  $S(\mu)$ is a seperable element in $\C[P_{\alpha}]$. Further the
terms of degree $d$
in the cofactor expansion of $S(\lambda)$ along $\alpha$ vanish
in the range
\[a_{\alpha, 1}(\lambda)>d> a_{\alpha,2}(\lambda)=a_{\alpha,
  1}(\lambda)-m_{\alpha}(\lambda).\]
The proof of the classical Eisenstein criterion now applies to
establish the lemma.
\end{proof}

\section{Key lemmas and the proof of the main theorem}
\label{sec:keyprops}

In this section we present the key propositions and the deduction of
the main theorem from these propositions.
The heart of the proof of the main theorem is the following
proposition:

\begin{proposition}\label{prop:key}
Let ${\mathcal R}=(E, \Phi, \Delta)$ be a simple based root
datum  of rank $l$, not isomorphic to $F_4$ or $G_2$.
 Let
$\mu$ be a regular weight in $P_{++}$ and $\eta=\mu+c\omega_{\alpha}$ for
some positive integer $c$, where  $\omega_{\alpha}$ is the fundamental weight
corresponding to $\alpha\in \Delta$. Assume that $\alpha$ is a corner
root in the Dynkin diagram for ${\mathcal R}$ if ${\mathcal R}$ is not
simply laced. Assume further that Theorem \ref{thm:main} is valid for
all simple Lie algebras of rank less than $l$.

Let  $e|d(\mu)$ and $d| e, ~e\neq d$ be natural
numbers. Suppose there exists symmetric non-unit elements $U, ~V\in \O[P]^W$
satisfying the following:
\begin{itemize}
\item \begin{equation}\label{eqn:keyprop1}
UV=C(\mu, d).
\end{equation}

\item   The factor $V$ of $C(\mu, d)$ divides $C(e\rho, d)$.

\item There exists elements $X, ~Y\in \O[P]$ such that
\begin{equation}
\label{eqn:keyprop2}
 UX+VY=C(\eta, d).
\end{equation}
\end{itemize}
Then the following holds:

\begin{enumerate}
\item If $\cR$ is simply laced or if $\alpha$ is a short root, then
$(e, c)>d$.

\item If $\cR$ is not simply laced and  $\alpha$ is a long root, then
$(e, m(\Phi)c)>d.$

\end{enumerate}

\end{proposition}

\begin{remark} Although the proposition can be shown to be valid
for  $F_4$ and   $G_2$,
for the proof of the main theorem, we will not require
  the $F_4$ and the $G_2$ cases of the foregoing proposition.
\end{remark}

\begin{remark}
When $\Phi$ is of rank one, the above proposition translates to a
statement about factors of a cyclotomic polynomial, such that their
combination is equal to another cyclotomic polynomial. The
proof of the proposition is arithmetic and is given in Section
\ref{sec:sl2}. Further, it is not required that the factors $U$ and
$V$ be invariant.

For higher ranks,  the arithmetical proof for $sl(2)$ does not
generalize, as the arithmetical properties of a
character, if any,  are not easy to understand.
The proof is by induction on the rank and is given
in Section \ref{sec:pfkeyprop}. The invariant condition on the factors
is required, to ensure that there is a corner root such that the
leading coefficients of $U$ and $V$ are not units
(see Proposition \ref{prop:nonmonic}).
\end{remark}

\subsection{Non-monic invariant factorizations}
In order to reduce the proof of the main theorem to that of
Proposition \ref{prop:key}, we need to rule out certain types of
factorizations: factorizations such that
in the cofactor expansion along any corner root, at
least one of the factors is monic. In the case of $GL(r)$ we are working with
symmetric homogeneous polynomials in $r$-variables.
We look at the polynomials as a
polynomial in $x_1$ with coefficients polynomials in the variables
$x_2, \cdots, x_r$.  In this case, we want to
rule out factorizations,  where either one of the factors is either monic or
the constant coeffecient is monomial. (we call such factorizations as  monic
factorizations, see also Example \ref{ex:glrleadingcoeff}).

Let $\cR=(E, \Phi, \Delta)$ be a based simple root system of rank
$r$.  Let $\alpha\in
\Delta$ be a corner root in the Dynkin associated to $E$.
\begin{definition}
An element
$U\in \C[P]$ is said to be {\em monic} with respect to $\alpha$, if there
is an unique weight $\mu_{\alpha}\in P(U)$ with maximum degree amongst
all the weights occurring in $U$, i.e.,
\[ \omega_{\alpha}^*(\mu_{\alpha})\geq \omega_{\alpha}^*(\mu) \quad
\forall \mu\in P(U)\] with equality if and only if
$\mu=\mu_{\alpha}$.
\end{definition}

If we further assume that $U \in  \C[P]^W$,
since the ring  $\C[P]^W$ is isomorphic to a
polynomial ring in $r$-variables, the element $\mu_{\alpha}$ is well
defined. By symmetry we observe that $\mu_{\alpha}$ is fixed by the
subgroup $W_{\alpha}$ of the Weyl group $W$ fixing
$\omega_{\alpha}$. Hence  $\mu_{\alpha}=u\omega_{\alpha}$ for some
integer $u$.

\begin{definition} \label{dfn:nonmonic}
Let $C$ be an element in $\C[P]$,
and suppose there is a factorization $C=UV$ in
$\C[P]$. The factorization is said to be {\em non-monic}, if there exists
a corner root $\alpha$
in the Dynkin diagram of $\cR$, such that both $U$ and $V$ are not monic
along $\alpha$.
\end{definition}

We show that any possible invariant  factorization is non-monic under
the assumption {\em NMFG:}
\begin{proposition} \label{prop:nonmonic}
Let $\cR=(E, \Phi, \Delta)$ be a simple based root
system of rank $r$. Assume that Theorem \ref{thm:main} is valid for
all simple root systems of rank strictly less than $r$. Let $\lambda$
be a dominant regular weight for $(E, \Phi, \Delta)$. Suppose there is
a factorization,
\[ C(\lambda)=UV,\] where both $U$ and $V$ are in $\C[P]^W$. Then
this factorization is non-monic, except
when the root system is of type  either $G_ 2$ or $F_4$ and the
weight $\lambda$ is of the form,
\[ \lambda=u\omega_{\alpha}+v\omega_{\beta}+d(\lambda)\rho,\quad d(\lambda)=(u,v)\]
where $\omega_{\alpha}$ (resp. $\omega_{\beta}$) is the fundamental
weight corresponding to the short (resp. long) corner root $\alpha$
(resp. $\beta$) in the Dynkin diagram such that the following
inequalities are satsified:
\[ m(\Phi)v \geq u+d(\lambda)\quad \text{and}\quad  u\geq v+d(\lambda).\]
\end{proposition}

\begin{remark}
Proposition \ref{prop:nonmonic} can also be considered as establishing
the irreducibility
property in $\C[P]^W$ for a `general´ weight $\lambda$, i.e., those for which
the leading coeffecient of  $C(\lambda^{\alpha}, d(\lambda))$ along
any corner root $\alpha$  is irreducible.
\end{remark}

\subsection{Proof of Main Theorem} We now prove Theorem
\ref{thm:main} assuming the validity of the
 above Propositions \ref{prop:key} and
\ref{prop:nonmonic}.  First of all by duality (see Remark
\ref{rem:duality-red-pf}), we can assume that $\lambda$ is  an element
of $m(\Phi)^iP\backslash m(\Phi)^{i+1}P^*$ for some $i\geq 0$.  In this case,
the greatest common divisor $D(\lambda)$ of the  `obvious' Weyl
denominator type factors of $S(\lambda)$ is $S(d(\lambda)\rho)$, and
we need to show that  $C(\lambda, d)= S(\lambda)/S(d(\lambda)\rho)$ is
irreducible, where $d=d(\lambda)$.

Suppose there is a factorization in $\C[P]^W$,
\begin{equation}\label{eqn:fact}
 C(\lambda, d)=QR.
\end{equation}
We have $C(\lambda, d) \in \Z[P]^W$ and the coefficient of
$e^{\lambda}$ is $1$. 

{\em Claim:} Given any natural number $N$, there exists a UFD 
${\mathcal O}\subset \C$ such that the rational primes $p\leq N$ are
not units in ${\mathcal O}$, and $Q$ and $R$ are in ${\mathcal
  O}[P]^W$ upto multiplying by constants. 

{\em Proof of Claim.} The claim goes under the name of Lefschetz
principle, and for the sake of completeness
 we give a proof based on Gauss' lemma: Suppose $A$ is a
UFD, and $C\in A[x_1, \cdots, x_n]$ is a polynomial such that the gcd
of its coefficients is a unit in $A$. Then if $C$ admits a
factorization in $K[x_1, \cdots, x_n]$ where $K$ is the quotient field
of $A$, then  it admits a factorization
in $A[x_1, \cdots, x_n]$.

To prove the claim, by attaching the coefficients of $Q$ and $R$, we
can assume that the factorization is over a finitely generated field
$E$ over $\Q$. Let $F$ be the algebraic closure of $\Q$ in $E$; this
is a finite extension of $\Q$ and $E$ is a finitely generated purely
transcendental extension over $F$. We can write $E$ as the quotient
field of $A=F[\xi_1, \cdots, \xi_k]$ for some algebraically
independent generators $\xi_1, \cdots, \xi_k$. By Gauss' lemma, we can
assume (after multiplying by some units) that the factors $Q$ and $R$
belong to  $A[x_1, \cdots, x_n]$. By considering them as polynomials
in the variables  $\xi_1, \cdots, \xi_k, ~x_1, \cdots, x_n$ and by
degree considerations, we see that the coefficients of $Q$ and $R$
have to lie in the number field $F$. 

Let ${\mathcal O}_F$ be the ring of algebraic integers in $F$. By
inverting the primes $q\geq N$, we get a semilocal ring ${\mathcal
  O}$. This is a UFD and the primes $p<N$ are not units. By Gauss'
lemma, the factorization is defined over ${\mathcal
  O}[P]^W$, and this proves the claim. 

Let $N$ be a natural number greater
than $2m(\Phi)(\sum_{\alpha\in \Delta}m_{\alpha}(\lambda))$.
We also assume that ${\mathcal O}$ contains the $N$-th roots of unity.

By Proposition \ref{prop:nonmonic},  choose a corner root $\alpha_0$ of the
Dynkin diagram of $\cR$ such that the leading coefficients
$U$ (resp. $V$) of $Q$ (resp. $R$) along $\alpha_0$ with respect to the
co-factor expansion of $Q$ and $R$ along $\alpha_0$ are not units in the
ring ${\mathcal O}[P_{\alpha_0}]^{W_{\alpha_0}}$.  Let
\[\mu=\lambda^{\alpha_0}.\]  Then
\begin{equation}
\C(\mu, d(\lambda)\rho)=UV.
\end{equation}
The main observation that in conjunction with  Proposition
\ref{prop:key} will allow us to prove Theorem \ref{thm:main} is the
following:
\begin{equation}\label{eqn:gcd}
d=d(\lambda)=(d(\mu), m_{\alpha_0}(\lambda))= (d(\lambda^{\alpha_0}),
d((s_{\alpha_0}\lambda)^{\alpha_0})).
\end{equation}

The cofactor expansion of $Q$ and $R$ is given by,
\begin{equation}\label{eqn:expn}
 Q=\sum_{i\geq 0}e^{(q-i)l_{\alpha-)}}Q_{\alpha_0, q-i}, \quad
R=\sum_{i\geq 0}e^{(r-i)l_{\alpha_0}}R_{\alpha_0, r-i},
\end{equation}
where we have denoted by $q$ (resp. $r$) the degrees of the cofactor
expansions of $Q$ and $R$ along $\alpha$. Further since $Q$ and $R$
are assumed to be invariant, the coefficients $Q_{\alpha_0, q-i}$ and
$R_{\alpha_0, r-i}$ are $W_{\alpha_0}$-invariant.

By Eisenstein criterion Lemma \ref{lem:eisencrit}, the leading
coefficient $U$ (resp. $V$) of $Q$ (resp. $R$) along
$\alpha_0$ divides $Q_{\alpha_0,q-i}$ (resp.  $R_{\alpha_0,r-i}$)
for $i <n_{\alpha_0}(\lambda)$. Upon substituting the cofactor
expansions of $Q$ and $R$ in  equation \ref{eqn:fact}, we see that
there exists elements $X, Y$ in the ring
${\mathcal O}[P_{\alpha_0}]^{W_{\alpha_0}}$ such that,
\begin{equation} UX+VY=C(\eta,d(\lambda)).
\end{equation}
Here $\eta$ is given by the formula in Lemma \ref{lem:cofexp}:
\begin{equation}
\eta=(s_{\alpha_0}\lambda)^{\alpha_0}=\mu+c\omega^{\alpha_0}_{\alpha},
\end{equation}
where $\alpha$ is the unique root in the Dynkin diagram attached to
$\cR$ that is connected to the corner root $\alpha_0$, and $c$ is
given by
\begin{equation}\label{eqn:c}
\begin{split}
c&=m_{\alpha_0}(\lambda)|m_{\alpha} (\alpha_0)|\\
&=\begin{cases} m(\Phi)m_{\alpha_0}(\lambda) & \text{if $\alpha_0$
    is long and $\alpha$ is short},\\
m_{\alpha_0}(\lambda) &\text{otherwise}.
\end{cases}
\end{split}
\end{equation}

We now analyze the possible factorizations of $C(\mu, d(\lambda))$
assuming  the validity of Theorem \ref{thm:main} for
$\cR_{\alpha_0}$, and apply Proposition \ref{prop:key} to the root system
$\cR_{\alpha_0}$ together with Equation \ref{eqn:gcd} to arrive at a
contraduction (and thus prove Theorem \ref{thm:main}).

{\em Case (i). } We first analyze the case where the root system
$\cR_{\alpha_0}$ is either simply laced, or if not simply laced, then
 $\mu$ does not belong to
$m(\Phi_{\alpha_0})^{i+1}P^*_{\alpha_0}$.

If $\mu$ is not of Weyl denominator type, then
\begin{equation}
C(\mu, d(\lambda))= C(\mu, d(\mu))C(d(\mu), d(\lambda)),
\end{equation}
where $C(\mu, d(\mu))$ is absolutely irreducible. In this case, let
\[ e=d(\mu),\quad d=d(\lambda), \quad \text{and} \quad  d|e.\]
We also assume that $V$ divides $C(e\rho, d)$.

If $\mu=e\rho$ is of
Weyl denominator type, then $V$ (and $U$ too)  divides $C(e\rho, d)$.

Now suppose that
 we are {\em not} in the case where $\alpha_0$ is long and $\alpha$
is a short root. Then, by Equation \ref{eqn:c},
\[ c=m_{\alpha_0}(\lambda).\]
Applying Proposition \ref{prop:key},
\[ d=d(\lambda)=(d(\mu), m_{\alpha_0}(\lambda))>d,\]
and this gives a contradiction.

Assume now that we are in the situation where $\alpha_0$ long,
$\alpha$ is short (and $\mu$ does not belong to
$m(\Phi_{\alpha_0})^{i+1}P^*_{\alpha_0}$).
In this case, the root system $\cR^{\alpha_0}$ is simply laced and
let $e=d(\mu), ~c=m(\Phi)m_{\alpha_0}(\lambda)$. By Proposition
\ref{prop:key},
\begin{equation}\label{eqn:longshort}
 (d(\mu), m(\Phi)m_{\alpha_0}(\lambda))>d.
\end{equation}
We concentrate only on the contribution by the prime number $m(\Phi)$
to the computation of the gcd's.  Since we have assumed that $\lambda\in
m(\Phi)^i\backslash m(\Phi)^{i+1}P^*$, this implies
\[d(\mu)=m(\Phi)^id'(\mu)\quad \text{with} \quad (d'(\mu),
m(\Phi))=1.\]
Write $d=m(\Phi)^id'$ with $d'$ coprime to $m(\Phi)$.
Let  $m_{\alpha_0}(\lambda)=m(\Phi)^im'_{\alpha_0}(\lambda)$ where
$m'_{\alpha_0}(\lambda)$ is a natural number.
Substituting we get the following couple of equations,
\begin{equation}
\begin{split}
(d'(\mu)m(\Phi)^i, m(\Phi)^{i+1}m'_{\alpha_0}(\lambda))& >m(\Phi)^id'\\
(d'(\mu), m_{\alpha_0}(\lambda))& =d',
\end{split}
\end{equation}
These equations imply that $m(\Phi)$ divides $d'(\mu)$ contradicting
the choice of $\lambda$.

{\em Case (ii).}
Now suppose that  the root system ${\mathcal R}_{\alpha_0}$ is not simply
laced and $\mu \in m(\Phi_{\alpha_0})^{i+1}P_{\alpha_0}^*$.
From the classification of the Dynkin diagrams, this can happen
only when both $\alpha_0$ and its neighbour $\alpha\in \Delta$ have
the same lengths. Further $m(\Phi)=m(\Phi_{\alpha_0})$.

Now if $\alpha_0$ is a long root, the assumption that
$\mu=\lambda^{\alpha_0}$ lies in $m(\Phi)^{i+1}P_{\alpha_0}^*$ implies
that $\lambda$ belongs to $m(\Phi)^{i+1}{P}^*$, a contradiction since we have
assumed that $\lambda$ is not an element of $m(\Phi)^{i+1}{P}^*$.

Hence we can
assume that both $\alpha_0$ and $\alpha$ are short roots. If
$m(\Phi)^{i+1}$ divides $m_{\alpha_0}(\lambda)$, then by our
assumption on $\mu$, we get that $\lambda \in m(\Phi)^{i+1}P^*$,
contradicting our initial choice of $\lambda$. Hence, we can write
\[ m_{\alpha_0}(\lambda))=m(\Phi)^im',\]
with $m'$ coprime to $m(\Phi)$.

We have a
factorization (assuming $\mu$ is not of Weyl denominator type)
\begin{equation}
C(\mu, d(\lambda))= C(\mu,d^*(\mu)\tilde{\rho})C(d^*(\mu)\tilde{\rho},
d(\lambda)),
\end{equation}
where $d^*(\mu)$ defined as in Equation \ref{eqn:d*}
is the largest integer such that
$d\tilde{\rho}$ divides $\mu$. The factor $ C(d^*(\mu)\tilde{\rho},
d(\lambda))$ divides $C(m(\Phi)d^*(\mu){\rho},
d(\lambda)) $. In this case, we take
\[ e=m(\Phi)d^*(\mu), \quad d=d(\lambda).\]
We assume that $V$ is not coprime to
$C(e\rho, d)$.

If $\mu$ is of Weyl denominator type then both the factors divide
$C(e\rho, d)$, where $e$ is defined as above.

By  Proposition \ref{prop:key},
\[ (m(\Phi)d^*(\mu),m_{\alpha_0}(\lambda)) >d.\]
Since $\mu\in m(\Phi)^{i+1}P_{\alpha_0}^*$, it follows that
$d^*(\mu)$ is divisible by $m(\Phi)^i$. We argue as above, and write
\[d^*(\mu)=d'(\mu)m(\Phi)^{i}, \quad d=m(\Phi)^id',\]
where $d'$ is  coprime to $m(\Phi)$. The above inequality can be
written as,
\[ (m(\Phi)d'(\mu),m') >d'.\]
On the other hand, Equation \ref{eqn:gcd} yields,
\[ d'=(d(\mu)/m(\Phi)^i, m').\]
Since $d^*(\mu)$ divides $d(\mu)$, it follows
that the only way this is possible
is if $m(\Phi)$ divides $m'$.   This
contradicts our choice of $\lambda$.

Hence Theorem \ref{thm:main} is proved modulo the proofs of
Propositions \ref{prop:key} and \ref{prop:nonmonic}.

\section{An arithmetic lemma} \label{sec:sl2}

In this section we give an arithmetical argument to establish
Proposition \ref{prop:key} when the root system is isomorphic to
$sl_2$. This turns out to be an arithmetic statement in the context of
polynomial ring in one variable over suitable rings.
It is this arithmetic statement
that is at the heart of the proof of the main theorem in the general
case. The proof of the main theorem is to reduce by induction on the
rank to the arithmetic statement (and not to the irreducibility of
characters say of $GL(3)$). In retrospect, as irreducibility is
connected with divisibility, it is to be expected that the proof of an
irreducibility result depends on some arithmetic.

We first recall the following elementary lemma from cyclotomic theory.
\begin{lemma} \label{lem:cyclo}
Let $e$ be a natural number and let $\zeta_e$ denote a
  primitive $e$-th root of unity.

i) If $e$ is composite, then $(1-\zeta_e)$ is a unit in the ring
 $\Z[\zeta_e]$.

ii) If $e=p^k$ for some prime number $p$, then
$(1-\zeta_e)$ divides $p$ in the ring  $\Z[\zeta_e]$. In particular
if $p|N$,  then $(1-\zeta_e)$ is not a unit in the ring
$\Z[\zeta_e]$; thus in any ring $\O$ containing $\Z[\zeta_e]$ in which
$p$ is not invertible.
\end{lemma}
\begin{proof} Let
\[ \Phi_e(x)=\frac{x^m-1}{x-1}=1+x+\cdots +x^{m-1}.\] We have
$\Phi_e(1)=e$. Further if $e$ and $f$ are coprime natural numbers,
then $\Phi_e$ and $\Phi_f$ are coprime polynomials.

For coprime  natural numbers $f_1, \cdots, f_k$ dividing  $e$, define
\[ \Phi_{e:f_1, \cdots, f_k}(x)= \Phi_e(x)/
\prod_{i=1}^k\Phi_{f_i}(x)=\prod_{\zeta^e=1,
  \zeta^{f_i}\neq 1}(x-\zeta), \]
where $i$ goes from $1$ to $k$ in the above product.
The gcd of the coeffecients of $\Phi_e(x)$ is $1$, and it is
clear that $\prod_{i=1}^k\Phi_{f_i}(x)$ divides $\Phi_e(x)$ in
the ring $\Q[x]$. Hence by Gauss's lemma,
$\Phi_{e:f_1, \cdots, f_k}(x)\in \Z[x]$.

Let $e=p_1^{r_1}\cdots p_k^{r_k}$ be the
factorization of $e$ into primes, where $p_i\neq p_j$ are mutually
distinct rational primes. Since $\Phi_{e:p_1^{r_1}, \cdots, p_k^{r_k}}(1)=1$,  i)
  follows from the above product decomposition for
$\Phi_{e:p_1^{r_1}, \cdots, p_k^{r_k}}(x)$.

Similarly,  ii) follows from the fact that
\[\Phi_{p^k:p^{k-1}}(1)=p.\]
\end{proof}

\begin{remark} The proof uses the fact that we are working over a
characteristic zero ring.
\end{remark}

Let $N$ be a natural number (to be chosen later depending on the
dominant regular $\lambda$), and let $\O$  be a ring of
characteristic zero such that the prime numbers $p\leq N$ are not
units in $\O$.

The basic arithmetic proposition is the following:
\begin{proposition}
\label{prop:arith}
Let $e,~f$ be natural numbers not bigger than $N$ and
divisible by
 a natural number  $d$.
Suppose there are non-unit elements $U, ~V\in \O[x]$ and
elements  $X, ~Y$ in $\O[x]$ satisfying the following:

\begin{eqnarray}
UV &= \Phi_{e:d}\\
 UX+VY&=\Phi_{f:d}
\end{eqnarray}
Then the greatest common divisor  $(e,f)$ of $e$ and $f$
is strictly greater than $d$.
\end{proposition}

\begin{proof}  For a natural number $e$ let
 ${\mu}_e$ denote the
  group of $e$-th roots of unity; for a rational prime $p$ let
${\mu}_{p^{\infty}}$ denote the group of roots of unity of order a
  power of $p$. By enlarging $\O$ by adjoining roots of unity, we can
  assume that $U$ and $V$ factorizes into linear factors in $\O[x]$.
Given a polynomial $W\in\O[x] $, let $Z_W$ denote the
  zeros of $W$ in $\O$. Upto units in $\O[x]$,
 there is a factorization
\[ U(x)=\prod_{\zeta\in Z_U}(1-\zeta^{-1} x), \quad V(x)=\prod_{\zeta\in
  Z_V}(1-\zeta^{-1} x),\]
so that
\begin{equation}\label{eqn:ed}
Z_U\cup Z_V={\mu}(e,d):=\{\zeta\mid \zeta^{e}=1, \zeta^{d}\neq 1\}.
\end{equation}
We first make the following claim:

{\em Claim}: there exists $\gamma\in Z_U$ (or $Z_V$) and $\delta\in
Z_V$ (resp. $\delta\in Z_U$) such that $\gamma^{-1}\delta$ is an
element of prime power order for some rational prime $p$.

\noindent{\em Proof of Claim}. Choose an element
$\gamma\in Z_U$, such that $\gamma$ can be expressed as a product
$\gamma= \gamma_1\cdots\gamma_k$ having the
following properties:
\begin{itemize}
\item $\gamma_i\in \mu_{p_i^{\infty}}$ for some rational primes $p_i$.
\item $p_i\neq p_j$ for $1\leq i, ~j\leq k$.
\item $k$ is minimal amongst all $\gamma\in Z_U$.
\end{itemize}
Since $Z_U$ is non-empty such a choice is possible by Chinese remainder
theorem.

Suppose for some $i$, $\delta:=\gamma\gamma_i^{-1}\in Z_V$.
 Since $\gamma_i$ is $p_i$-primary,
\[\gamma^{-1}\delta =\gamma_i^{-1}\]
is $p_i$-primary, and this proves the claim in this case.

Hence we can assume that for any $i=1,\cdots, k$,
$\gamma\gamma_i^{-1}\in Z_U\cup
{\bf \mu}_d$. But since $k$ is minimal, this implies that
\[\gamma\gamma_i^{-1}\in {\bf \mu}_d\]
for every $i=1,\cdots, k$.
  If $k>1$, then this implies that
$\gamma\in {\bf \mu}_d$, contradicting the fact that $Z_U$ is coprime
to $\Phi_d$.  Hence we see that there exists an element
$\gamma_p\in \mu_{p^{\infty}}\cap Z_U$ for some rational prime $p$.

Similarly arguing with $Z_V$,
we obtain an element $\delta_q\in \mu_{q^{\infty}}\cap Z_V$
for some rational prime $q$.

If $q=p$, then take $\gamma=\gamma_p$ and
$\delta=\delta_q$. This establishes the claim since the polynomial
$\Phi_{e}$ is separable, hence $\gamma^{-1}\delta\neq 1$.

Assume now that $p\neq q$. It follows that the element
$\gamma_p\delta_q$ belongs to the set ${\mu}(e, d)$ defined as in
Equation \ref{eqn:ed}. If it  belongs to $Z_U$
(resp. $Z_V$), then the pair  $\gamma =\gamma_p\delta_q$
(resp. $\delta=\gamma_p\delta_q$) and $\delta=\delta_q$
(resp. $\gamma=\gamma_p$) produces the elements as required by the
claim. This proves the claim.

Now we deduce the proposition from the claim. Upon substituting $x=\delta$,
we get
\begin{equation}\label{eqn:imp}
U(\delta)X(\delta)+V(\delta)Y(\delta)=\Phi_{f,d}(\delta).
\end{equation}
We have $V(\delta)=0$ and
\[U(\delta)= \prod_{\zeta\in Z_U}(1-\zeta^{-1} \delta).\]
By the claim, there exists a factor of the form
$(1-\gamma^{-1}\delta)$ of $U$, such that $\gamma^{-1}\delta\in {\bf
  \mu}_{p^{\infty}}$ for some rational prime $p$. By Part ii) of Lemma
  \ref{lem:cyclo}, this implies that $U(\delta)$  is is not a unit in
  $\O$.  So the
left hand side of Equation \ref{eqn:imp} is not a unit.
Now,
\[ \Phi_{f: d}(\delta)=\prod_{\zeta\in {\mu}(f, d)}
(1-\zeta^{-1} \delta).\]
Suppose
$(e,f)=d$. Then for any pair of divisors $e'|e, ~f'|f$ with $d$ dividing
$e'$ and $f'$ and not equal to  either $e'$ or $f'$,
the least common multiple of $e'$ and
$f'$ has at least two distinct prime factors.  Hence for any
$\zeta\in {\mu}(f, d)$ of order
$f'$ (take $\delta$ to be of order $e'$), the factor $(1-\zeta^{-1}
\delta)$ is a unit. Hence  $\Phi_{f, d}(\delta)$ is a unit, contradicting
Equation \ref{eqn:imp},  and this proves the proposition.
\end{proof}

\begin{remark} The proof of Proposition \ref{prop:arith} given out here uses
  characteristic zero methods. Consequently,
 the proof of the absolute irreducibility of characters (Theorem \ref{thm:main})
  given in this paper
 does not carry over to positive characteristics, even for those
  admissible weights for which the Weyl character formula is known
  to be valid.
\end{remark}

\begin{remark} \label{rmk:idea}
Let $\lambda=(a_1>a_2>0)$  be a regular weight for
  $GL(3)$, with gcd $(a_1, a_2)=d$.  Consider a factorization $C(\lambda)=UV$
such that the  leading coefficients $U_u$ and $V_v$ respectively
of $U$ and $V$ considered as a
polynomial in the variable $x_1$ with coefficients polynomials in
$x_2, ~x_3$ are not monomials.
 A little argument using divisibility as used in the proof
of Eisenstein criterion, yields a pair of equations of the form,
\[ U_uV_v=(x_2^{a_2}-x_3^{a_2})/(x_2^{d}-x_3^{d}),  \quad
U_uX+V_vY=(x_2^{a_1}-x_3^{a_1})/(x_2^{d}-x_3^{d}),\]
for some polynomials $X$ and $Y$. Proposition \ref{prop:arith} applies
to give a contradiction. The entire schema of this paper is built
around this proof.
\end{remark}

 For the root system given by  $sl_2$,  the fundamental weight is
 given by $\rho$.   The ring $\C[P]$ can be identified
with the ring of Laurent polynomials $\C[x,x^{-1}]$ by substituting
$x=e^{\rho}$.  For any pair of natural numbers $e, ~d$ with $d|e$, the
element $ C(e\rho, d)$ can be written as,
\begin{equation}\label{eqn:laurenttoord}
C(e\rho, d)=
\frac{x^e-x^{-e}}{x^d-x^{-d}}=x^{e-d}\frac{x^{2e}-1}{x^{2d}-1}
=x^{e-d}\Phi_{2e, 2d}.
\end{equation}

\begin{corollary} \label{cor:key-sl2}
Proposition \ref{prop:key} is valid for the root
  system given by the Lie algebra $sl_2$, i.e.,
let $e, ~c, ~f=e+c$ be natural numbers and let
$\mu=e\rho,~\eta=f\rho=(e+c)\rho$ be dominant weights for
the root system $A_1$.
 Let $d$ be a natural number dividing $e,~c$.
Suppose there are non-unit  elements $U, ~V\in \O[x,x^{-1}]$ and
elements  $X, ~Y$ in $\O[x, x^{-1}]$ such that the following pair of
equations are satisfied:
\begin{align} \label{eqn:coeff3}
 UX+VY& =C(f\rho,d),\\
C(e\rho, d)&=UV.
\end{align}
Then $(e, c)=(e,f)>d$.
\end{corollary}

\begin{proof}
From Proposition
\ref{prop:arith}, and by Equation \ref{eqn:laurenttoord}, we get
$(2e, 2f)=2(e,f)>2d$. Hence it follows that $(e,f)>d$.
\end{proof}

\begin{remark} We do not actually require that $U$ and $V$ are
  $W$-invariant in the above corollary. The invariant hypothesis is
  required in the induction step of the proof of Proposition
  \ref{prop:key} at the stage when we use Proposition
  \ref{prop:nonmonic}.
\end{remark}

\section{Proof of Proposition \ref{prop:key}}\label{sec:pfkeyprop}
In this section we give a proof of  Proposition
\ref{prop:key} for a simple based  root system $\cR=(E,\Phi, \Delta)$
of rank $l<r$, needed
for the proof of the main theorem for a simple root system of rank
$r$.  In the
rank one case, Proposition \ref{prop:key} is proved as Corollary
\ref{cor:key-sl2} in the last section using arithmetic
methods. In contrast the proof in the higher rank case proceeds by
 induction on the rank; we reduce to a lower
rank situation using Proposition \ref{prop:nonmonic}.

A fundamental ingredient in the proof of  Proposition \ref{prop:key} is
the `universal divisibility' of the Weyl denominator as in Proposition
\ref{prop:weyldenominator}. This is applied
in the following manner:
\begin{lemma}
\label{lem:div}
Let $d$ be a natural number and let $U$ be a factor of
either   $S(d\rho)$ or  $S(d\tilde{\rho})$.
Suppose $\alpha$ is a corner root. Consider the cofactor
  expansion of $U$ along $\alpha$,
\[U=\sum_{i\geq 0}e^{(u-i)l_{\alpha}}U_{\alpha, u-i}.\]
Then the leading coefficient $U_{\alpha, u}$
divides $U_{\alpha, u-i}$ for all $i\geq 0$.
\end{lemma}

\begin{proof} It is enough, by duality,
 to prove the lemma  when $U$ is a factor of $S(d\rho)$. The leading
 coefficient of $S(\rho)$ along $\alpha$ is given by
 $S(\rho^{\alpha})$, which is the Weyl denominator for the root system
 $\cR_{\alpha}$. Scaling by $d$, implies that the leading coefficient
 $S(\rho^{\alpha})$ of $S(d\rho)$ along $\alpha$ divides all the
 coefficients in the cofactor expansion of $S(d\rho)$ along
 $\alpha$. Now we apply the proof of the Eisenstein criterion to
 establish the lemma for any factor $U$ of $S(d\rho)$.
\end{proof}

We now begin the proof of Proposition \ref{prop:key}.
 The proposition   is true for the rank one root system by
Corollary \ref{cor:key-sl2}.
By induction, assume that the
proposition holds for all simple root systems of rank less than $l$.
 By Proposition \ref{prop:nonmonic} (this is where we use the fact
 that the factors are $W$-invariant; see also Remark
 \ref{rmk:nonmonicuse} at the end of this section)
there is a corner root $\beta$ of the Dynkin diagram of $\Phi$ such
that the leading coefficient $U_{\beta, u}$ (resp. $V_{\beta, v}$) of $U$
(resp. $V$) in the cofactor  expansion
along $\beta$ is not monic.

Suppose $\beta\neq \alpha$. The root $\alpha$ considered as an element
in the root system $\cR_{\beta}$ continues to have the same property
regarding length, either short or long, as $\alpha$ has. We have,
\[ \eta^{\beta}=\mu^{\beta}+c\omega_{\alpha}^{\beta},\]
where $\omega_{\alpha}^{\beta}$ denotes now the fundamental weight
corresponding to $\alpha$ for the root system $\cR_{\beta}$.
Comparing the leading terms in
Equations \ref{eqn:keyprop1} and  \ref{eqn:keyprop2}, we get
\[\begin{split}
 C(\mu^{\beta},d)&= U_{\beta, u}V_{\beta, v}\\
U_{\beta, u}X_{\beta, x}+V_{\beta, v}Y_{\beta, y}& =C(\eta^{\beta}, d),
\end{split}
\]
where $X_{\beta, x}$ and $Y_{\beta, y}$ are the leading coefficients of $X$ and $Y$
respectively.  Since the leading coefficient of $C(e\rho, d)$ along
$\beta$ is given by $C(e\rho^{\beta}, d)$ and $V_{\beta,v}$ divides it,
we obtain by the  induction hypothesis that
$(e,c)>d$ if either $\cR_{\beta}$ is simply laced or $\alpha$ is a
short root,  and $ (e, m(\Phi_{\beta})c)>d$ if $\cR_{\beta}$ is not
simply laced and
$\alpha$ is a long root.
This establishes Proposition \ref{prop:key}, since
$m(\Phi_{\beta})=m(\Phi)$
if they are non-trivial.

Now suppose $\beta=\alpha$. In this case, we get
\[\eta^{\alpha}=\mu^{\alpha} \quad \text{and}\quad
C(\mu^{\alpha},d)=C(\eta^{\alpha},d)=U_{\alpha,u}V_{\alpha,v}. \]
and it looks as if it is impossible to set up the inductive
process. The trick out here is to observe that the universal
divisibility of the Weyl denominator, together with the fact that $V$
is of Weyl denominator type allow us to perform the inductive step by
considering the expansion upto the second leading non-zero term of
$S(\eta)$.

Considering the leading terms, we get the following equations:
\begin{equation}
\begin{split}  C(\eta^{\alpha},d)& =U_{\alpha,u}V_{\alpha,v}\\
 C(\eta^{\alpha}, d)& =U_{\alpha,u}X_{\alpha,x}+V_{\alpha,v}Y_{\alpha,y}.
\end{split}
\end{equation}
The separability of $C(\eta^{\alpha},d)$ implies that $U_{\alpha,u}$ and
$V_{\alpha,v}$ are coprime. Hence the above equations imply that
$V_{\alpha,v}$ divides $X_{\alpha,x}$.

We continue comparing the coefficients in the cofactor expansion along
$\alpha$ of the equation,
\[S(\eta)= (UX+VY)S(d\rho).\]
Let $T=S(d\rho)$.
Denote by $s$ (resp. $u,~v, ~x,~y, ~t$) the degrees of $S(\eta)$
(resp. $U, ~V, ~X,
~Y, ~T$) with respect to the cofactor expansion along $\alpha$. Here
$s=\omega_{\alpha}^*(\eta)$, and we have
\[ \omega_{\alpha}^*(\eta)-\omega_{\alpha}^*(s_{\alpha}(\eta))=
m_{\alpha}(\lambda).\]
Upon equating the $l_{\alpha}^{s-a}$-degree term in the
cofactor expansion of $S(\eta)$ along $\alpha$, we get
\begin{equation}
\label{exp:seta}
 \sum_{i+j+k=a}U_{u-i}X_{x-j}T_{t-k}
+ \sum_{l+m+n=a}V_{v-l}Y_{y-m}T_{t-n}= S({\eta})_{s-a},
\end{equation}
where we have suppressed the use of the subscript $\alpha$.
The coefficients $S({\eta})_{s-a}$ of $l_{\alpha}^{s-a}$-degree term in the
cofactor expansion of $S(\eta)$ along $\alpha$ are given by,
\begin{equation}
S(\eta)_{s-a}=
\begin{cases} S(\eta^{\alpha}) & \text{if $a=0$},\\
0 & \text{if $0<a < m_{\alpha}(\eta)$},\\
S((s_{\alpha}\eta)^{\alpha})  & \text{if
  $a=m_{\alpha}(\eta)$}.
\end{cases}
\end{equation}
To start the induction, we have $V_v$ divides $X_x$.
By induction assume that
$V_v$ divides $X_{x-j}$ for $j<a$.
Now suppose $0<a< m_{\alpha}(\eta)$
The right hand side in
Equation \ref{exp:seta} is zero, and Equation \ref{exp:seta} gives,
\[ 0= X_{x-a}U_uS(d\rho^{\alpha})+ \sum_{i+j+k=a, ~i>0}U_{u-i}X_{x-j}T_{t-k}
+ \sum_{l+m+n=a}V_{v-l}Y_{y-m}T_{t-n}.\]
Since $V$ is of Weyl denominator type,
by Lemma \ref{lem:div}, the
leading coefficient $V_v$ divides all the coefficients
$V_{v-j}$ for all $j\geq 0$ in the cofactor expansion of $V$ along $\alpha$.
Further in the second sum on the right, the assumption that $i>0$
implies that $j<a$.
Hence by induction and the fact
that $V_v$ is coprime to $S(d\rho^{\alpha})$ and $U_u$ we get that
$V_v$ divides $X_{x-a}$ for $a<m_{\alpha}(\eta)$.

Thus
upon comparing the  $s-a=m_{\alpha}(\eta)$ term in Equation
\ref{exp:seta} we get
\[ (U_uX_{x-a}+V_v \tilde{Y})S(d\rho^{\alpha})=S((s_{\alpha}\eta)^{\alpha}), \]
for some element $\tilde{Y}\in \O[P^{\alpha}]$. Here we have used the
fact that the leading coefficient $S(d\rho^{\alpha})$ divides all the
other coefficients $T_{t-j}$ in the cofactor expansion of $S(d\rho)$.
Now, by Equation \ref{eqn:salpha}
\begin{equation}
(s_{\alpha}\eta)^{\alpha} =\eta^{\alpha}+c_n\omega_{\alpha_n}^{\alpha},
\end{equation}
where $\alpha_n$ is the unique element of $\Delta$ that is connected
to the corner root $\alpha$ (since $\alpha=\beta$ is a corner root).
The value of $c_n$ is given by,
\begin{equation}
\begin{split}
c_n& = |<\alpha_n^*, \alpha>|m_{\alpha}(\eta)\\
&=\begin{cases} m(\Phi)(m_{\alpha}(\mu)+c) & \text{if $\alpha$ is long
      and $\alpha_n$ is short},\\
m_{\alpha}(\mu)+c  &\text{otherwise}.
\end{cases}
\end{split}
\end{equation}
The projection $\omega_{\alpha_n}^{\alpha}$ of the fundamental weight
corresponding to $\alpha_n$ in $\cR$ is the fundamental weight
corresponding to the corner root $\alpha_n^{\alpha}$ in
the root system $\cR_{\alpha}$.

Hence we can apply the  induction hypothesis. Note, by assumption
$e|m_{\alpha}(\mu)$. If $\alpha$ is long and $\alpha_n$ is short, then
the root system $\cR_{\alpha}$ is simply laced. By the inductive
hypothesis, we obtain,
\begin{equation}
d<(e, c_n)=(e,m(\Phi)(m_{\alpha}(\mu)+c))= (e, m(\Phi)c),
\end{equation}
establishing Proposition \ref{prop:key} in this case.

If $\cR_{\alpha}$ is not simply laced and $\alpha_n$ is a long root,
the inductive hypothesis again yields,
\[d<(e, c_n)=(e,m(\Phi_{\alpha})(m_{\alpha}(\mu)+c))= (e, m(\Phi)c)\]
since $m(\Phi_{\alpha})=m(\Phi)$.

Finally, if either $\cR_{\alpha}$ is simply laced or
$\alpha_n^{\alpha}$ is a short root in $\cR_{\alpha}$, then the
inductive hypothesis gives,
\[d<(e, c_n)=(e,(m_{\alpha}(\mu)+c))= (e,c).\]
This proves  Proposition \ref{prop:key}.

\begin{remark} \label{rmk:nonmonicuse}
One can avoid the use of Proposition
  \ref{prop:nonmonic} in this proof, by trying to prove directly the
  required statement.  Indeed in the simply laced case, the leading
  coefficients of a non-constant invariant $V$ dividing a generalized Weyl
  denominator function $S(d\rho)$, can easily seen to be non-trivial
  along any corner root. It is here that we require the factor $V$ to
  be invariant. For the other term $U$, if it contains a factor not of
  Weyl denominator type, it is not too
  difficult to see that there has to be at least one corner root along
  which the leading coefficient is non-unit. It seems  possible to extend
  this argument to cover the non-simply laced cases too, and avoid
  using the more general Proposition \ref{prop:nonmonic}.
\end{remark}

\section{Non-existence of invariant  monic
  factorizations}\label{sec:nonmonic}
Our aim in this section is to prove Proposition
\ref{prop:nonmonic}.  For the proof, we require a coprimality result
given by Proposition \ref{prop:coprime}, which we will prove in
Section \ref{sec:unique} as a corollary of Theorem
\ref{thm:unique} establishing an uniqueness property of $C(\lambda)$.
The proof of Proposition \ref{prop:nonmonic}
is essentially based on the use of
Eisenstein criterion (see Lemma \ref{lem:eisencrit}),  together with
Proposition \ref{prop:coprime} to obtain lower bounds for the degrees along the
corner roots of the factors $U$ and $V$. These bounds suffice except
for a class of weights for $F_4$ and $G_2$ (Assumption {\em NMFG}).

\subsection{Non-existence of symmetric monic factorizations for $GL(r)$}
We first give the proof for $GL(r)$. We restate Proposition \ref{prop:nonmonic} in the context of $GL(r)$:
\begin{proposition} \label{prop:nonmonicglr}
Let $\lambda=(a_1, \cdots, a_{r-1}, 0)$ be a normalized
  highest weight for $GL(r)$.
Suppose there is a factorization
\[ C(\lambda, d(\lambda))=UV, \quad U, ~V\in \C[x_1, \cdots, x_r]^W.\]
Write
\[ U(x_1, \cdots, x_r)= x_1^uU_u+x_1^{u-1}U_{u-1}+\cdots +
(x_2\cdots x_r)^{u_0}U_0,\]
\[ V(x_1, \cdots, x_r)= x_1^vV_v+x_1^{v-1}V_{v-1}+\cdots +
(x_2\cdots x_r)^{v_0}V_0,\]
where $U_u, \cdots, U_0$ and $V_v, \cdots, V_0$ are polynomials in the
variables $x_2, \cdots, x_r$.
Then either $U_u$ and $V_v$ or $U_0$
and $V_0$ are both non-constant polynomials.
\end{proposition}

\begin{proof} We use the notation as given in Section \ref{sec:cofactor}
  for the cofactor expansion of $GL(r)$. Let $d=d(\lambda)$.
We have
\[ UVS(d\rho)=S(\lambda).\]
Suppose $C(\lambda^{(1)}, d)$ and $C(\lambda^{(r-1)}, d)$ are
both constant polynomials. Then $\lambda=d\rho$ and there is
nothing to prove. We assume that $C(\lambda^{(1)}, d)$
is non-constant (and a similar argument can be given if we assume that
$C(\lambda^{(r-1)}, d)$ is non-constant).

i) Assume $V_v$ is constant.  Since $S_{\lambda^{(1)}}$ is separable, by
Eisenstein criterion, $U_u$ divides $U_u, \cdots,
U_{u-(a_1-a_2)+1}$. Now if $u-(a_1-a_2)+1\leq 0$, this implies $U_u$
divides $U$, and hence $U_u$ divides $S_{\lambda^{(2)}}$. But by
Proposition \ref{prop:coprime},
  $U_u=C(\lambda^{(1)},d)$ is coprime to
$C(\lambda^{(2)}, d)$, and this leads to a contradiction. Hence,
\begin{equation} \label{e1l1}
 u\geq a_1-a_2.
\end{equation}
Since $u+v+d(r-1)=a_1$, we get
\begin{equation}\label{e2l1}
 v\leq a_2-d(r-1).
\end{equation}
Since $V_v$ is constant by $W$-invariance
\begin{equation}\label{e3l1}
 V_0=(x_2^v+\cdots+x_r^v)+ \mbox{lower degree terms}.
\end{equation}
Hence $V_0$ is non-constant. If $U_0$ is also non-constant, then we
are done. Hence we can assume that $U_0$ is constant. Then,
\begin{equation}\label{e4l1}
v={\rm deg}_{x_2}(V_0)={\rm deg}_{x_2}(C(\lambda_0,d))=(a_1-a_{r-1})-d(r-2).
\end{equation}
Hence we get
\begin{equation}\label{e5l1}
 a_1-a_2\leq a_{r-1}-d.
\end{equation}

Similarly arguing with the constant term, we get $V_0$ divides $V_0, \cdots,
V_{a_{r-1}-1}$. By the coprimality of $C(\lambda^{(r-1)}, d)$ and
$C(\lambda^{(r-2)}, d)$, we get
\begin{equation}\label{e6l1}
 v\geq a_{r-1}.
\end{equation}
Hence
\begin{equation}\label{e7l1}
  u\leq (a_1-d(r-1))-a_{r-1}=a_1-a_{r-1}-d(r-1).
\end{equation}
But
\begin{equation}\label{e8l1}
{\rm deg}_{x_2}(U_u)={\rm deg}_{x_2}(C(\lambda^{(1)},
d))=a_2-d(r-2).
\end{equation}
By $W$-equivariance,
\begin{equation}\label{e9l1}
 u\geq a_2-d(r-2).
\end{equation}
Hence we get,
\[ a_2-d(r-2)\leq a_1-a_{r-1}-d(r-1)\]
\begin{equation}\label{e10l1}
a_{r-1}\leq (a_1-a_2)-d.
\end{equation}
Combining the above inequalities, we get
\[ a_{r-1}\leq (a_1-a_2)-d\leq a_{r-1}-2d, \]
clearly a contradiction. This proves the proposition.
\end{proof}

\subsection{Proof for D, E}
The proof of the proposition is easiest for the root systems of type
$D, ~E$ since there are more than two corner roots.
We first observe a simple fact about
polynomial rings $R[x]$ in one variable over an integral domain $R$:
the degree of a polynomial is the sum of the degrees of it's
factors. Applying this to the ring $\C[P]^W$, since $\lambda$ has
maximal degree in the cofactor expansion along any root $\alpha\in
\Delta$, amongst all the weights occurring in $S(\lambda)$, we
have:
\begin{lemma}\label{lem:uniqlambdadecomp} Suppose
$S(\lambda)=UVS(d(\lambda)\rho)$ where $U$ and $V$ are $W$-invariant.
Assume that there are weights $\mu\in P(U),~\nu\in P(V)$ such that
\[\lambda-d(\lambda)\rho=\mu+\nu.\]  Then $\mu$ (resp. $\nu$) has
maximal degree in $P(U)$ (resp. $P(V)$) with respect to any corner
root in $\Delta$.

In particular, if $U$ is monic along $\alpha$, then $\mu$ is the
unique weight in $P(U)$ having maximal degree along $\alpha$ and
$\mu=u\omega_{\alpha}$ for some integer $u$.
\end{lemma}

\begin{corollary}\label{cor:nonmonicsym}
 Suppose $U$ is an invariant factor of $S(\lambda)$
  which is monic along a corner root $\alpha$. Then $U$ cannot be
  monic along a different corner root $\beta$.
\end{corollary}
Indeed, the unique `highest weight' $\mu\in P(U)$ cannot
simultaneously be a multiple of $\omega_{\alpha}$ and
$\omega_{\beta}$. This corollary expresses the fact used in the proof
of Proposition \ref{prop:nonmonicglr},  that if a
symmetric homogeneous polynomial in at least two variables is monic
considered as a polynomial in one variable, then its constant term
cannot be a monomial.

\begin{corollary} Proposition \ref{prop:nonmonic} is true for the
simple root systems of type $D$ and $E$.
\end{corollary}
\begin{proof}  Since there are three corner roots for the root
systems of type $D, ~E$, this implies that there is at least one
corner root at which both $U$ and $V$ are not monic.
\end{proof}

\subsection{Non-simply laced root systems} From now onwards we
 consider a
non-simply laced based root systems $\cR=(E, \Phi, \Delta)$.
By Corollary \ref{cor:nonmonicsym}, we have
\begin{corollary}\label{cor:moniclambda}
Suppose
$S(\lambda)=UVS(d(\lambda)\rho)$ where $U$ and $V$ are
$W$-invariant. Assume further that $U$ is monic along one of the
corner roots $\alpha$ and $V$ is monic along the other corner root
$\beta$ in the Dynkin diagram associated to $\cR$. Then,
\begin{equation}
\lambda=u\omega_{\alpha}+v\omega_{\beta}+d(\lambda)\rho,
\end{equation}
with $d(\lambda)=(u+d(\lambda), v+d(\lambda))$.
\end{corollary}
\begin{proof}
If $U$ is symmetric and monic  along a corner root $\alpha$,  by
Corollary \ref{cor:nonmonicsym},
 $U$ cannot be monic
along the other corner root, say $\beta$. If $V$ is also monic along
$\alpha$, then Proposition \ref{prop:nonmonic} is true. Hence we can
assume that $V$ is monic along $\beta$. Assume that $U$ (resp. $V$)
has an unique maximal
weight $u\omega_{\alpha}$ (resp. $v\omega_{\beta}$).  This implies
that the weight $\lambda$ can
be written as
$\lambda=u\omega_{\alpha}+v\omega_{\beta}+d(\lambda)\rho$.
\end{proof}

We recall the following fact \cite[Exercise 5, Section 13, page
72]{H}, and it's consequences of relevance to us:

\begin{lemma} \label{lem:minus1} Let $\cR=(E, \Phi, \Delta)$ be a simple
based root system.

 \begin{enumerate}
\item Suppose $\cR$ is of type $A_1, ~B_r, ~C_r, ~F_4, ~G_2$. Then  $-1$
is an element of the Weyl group of $\cR$.

\item Suppose $\cR$ is of type $B_2, ~C_2, ~F_4, ~G_2$ and $\alpha$ be
any corner root in the Dynkin diagram associated to $\cR$. Then $-1$  is
an element of the Weyl group of $\cR_{\alpha}$.

\item If $r\geq 3$, then  there is a corner root in the Dynkin diagram
associated to $\cR=B_r$ (resp. $C_r$) such that $\cR_{\alpha}$ is again of
type $B$ (resp. $C$). In particular  $-1$  is an element of the Weyl
group of $\cR_{\alpha}$.
\end{enumerate}
\end{lemma}

For $\alpha, ~\beta\in
\Delta$, let
\begin{equation}\label{walphabeta}
w_{\alpha\beta}=<\omega_{\alpha}^*,\omega_{\beta}>.
\end{equation}
This quantity is independent of the $W$-invariant inner product on
$E$.
The consequence of the first part of the foregoing lemma is  that the
weights occurring in $S(\lambda)$ are invariant with respect to the
map $p\mapsto -p, ~p\in E$. An application of the proof of
Eisenstein's criterion yields  the trivial estimate:
\begin{lemma}\label{lem:monic-trivial-est}
With notation as in Corollary  \ref{cor:moniclambda}, the following holds:
\begin{equation}\label{eqn:trivialest}
 2uw_{\beta\alpha}+1~\geq ~<\beta^*, \lambda>.
\end{equation}
\end{lemma}
\begin{proof}
 We have the cofactor expansion of $U$ along $\beta$:
\begin{equation}\label{eqn:Ucofexp}
U=\sum_{j=0}^{2uw_{\beta\alpha}}e^{(uw_{\beta\alpha}-j)l_{\beta}}U_{\beta,
  uw_{\beta\alpha}-j}.
\end{equation} Here we have used the symmetry of $U$, the fact that
$-1$ belongs to the Weyl group, to obtain that
the degree along $\beta$ of the weights in $U$ varies from
$uw_{\beta\alpha}$ to $-uw_{\beta\alpha}$, since the weight with
maximum degree along $\beta$ is given by $u\omega_{\beta}$. From the
proof of the Eisenstein criterion, we get that
$U_{\beta, uw_{\beta\alpha}}$ divides the terms
\[ U_{\beta,uw_{\beta\alpha}}, \cdots, U_{\beta, uw_{\beta\alpha}-<\beta^*,
\lambda>+1}.\]
Consequently, if $2uw_{\beta\alpha}+1$ is less
than $<\beta^*, \lambda>$, then  $U_{\beta,uw_{\beta\alpha}}$
divides all the coefficients of $U$ along $\beta$, and hence all the
coefficients of $S(\lambda)$ along $\beta$. Since
$U_{\beta,uw_{\beta\alpha}}$ is not a unit, this contradicts the
corpimality of $S(\lambda^{\beta})$ and
$S((s_{\beta}\lambda)^{\beta})$. Hence the lemma follows.
\end{proof}

We can prove a sharper estimate,  assuming that $-1$ belongs to the
Weyl group of $\cR_{\alpha}$:
\begin{lemma}\label{lem:monic-sharp-est} With notation as above,
assume that the corner root $\alpha$ is such that the  automorphism
$x\mapsto -x$ is an element of the Weyl group of the root system
$\cR_{\alpha}$.  Then
\[ vw_{\alpha\beta}~\geq~ <\alpha^*, \lambda>.\]
\end{lemma}
\begin{proof} The  hypothesis that $-1$ belongs to the Weyl group of
$\cR_{\alpha}$, implies that the leading coefficient
$S(\lambda^{\alpha})$ of $S(\lambda)$ along $\alpha$ is mapped to
itself by the inverse map $p\mapsto -p$ of $\cR_{\alpha}$.  We have the
cofactor expansion of $V$ along $\alpha$:
 \begin{equation}\label{eqn:Vcofexp}
V=\sum_{j=0}^{2vw_{\alpha\beta}}e^{vw_{\alpha\beta}-jl_{\alpha}}V_{\alpha,
  vw_{\alpha\beta}-j}.
\end{equation}
Since $V$ is invariant by $W(\cR)$ and $-1\in W(\cR)$, the
term $e^{jl_{\alpha}}V_{\alpha,j}$ goes to the term
$e^{-jl_{\alpha}}V_{\alpha, -j}$ by the map $p\mapsto -p$ on
$P$. Since $U$ is monic along $\alpha$, the top degree term
$V_{\alpha, vw_{\alpha\beta}}$ is equal to
$S(\lambda^{\alpha})/S(d(\lambda)\rho^{\alpha})$. By our hypothesis
that $-1$ belongs to the Weyl group of $\cR_{\alpha}$, we find that
$V_{\alpha, vw_{\alpha\beta}}=V_{\alpha, -vw_{\alpha\beta}}$.  From
the proof of Eisenstein's criterion, we have that
$V_{\alpha, vw_{\alpha\beta}}$ divides
\[V_{\alpha, vw_{\alpha\beta}}, \cdots,
V_{\alpha, vw_{\alpha\beta}-<\alpha^*, \lambda>+1}.\]
Similarly, since $S(\lambda)$ is symmetric ($-1\in W$), we find that
$V_{\alpha, vw_{\alpha\beta}}$ divides the coefficients
\[V_{\alpha, -vw_{\alpha\beta}}, \cdots,
V_{\alpha, -vw_{\alpha\beta}+<\alpha^*, \lambda>-1}.\]
If $vw_{\alpha\beta}$ is less that $<\alpha^*, \lambda>$, this implies
that $V_{\alpha, vw_{\alpha\beta}}$ divides all the coefficients of
$V$ and hence of $S(\lambda)$ in the cofactor expansion of
$S(\lambda)$ along $\alpha$. But this contradicts the corpimality of
$S(\lambda^{\alpha})S(d(\lambda)\rho^{\alpha}$ and $
S((s_{\alpha}\lambda)^{\alpha})/S(d(\lambda)\rho^{\alpha})$. This
establishes the lemma.
\end{proof}

We compute the numbers $w_{\alpha\beta}$ explicitly:
\begin{lemma}\label{lem:walphabeta}
Let $\cR=(E, \Phi, \Delta)$ be a
non-simply laced simple based root system. Let $\alpha, ~\beta$ be
corner roots in the Dynkin diagram associated to $E$. Then the
following holds:

\begin{enumerate}
 \item $w_{\alpha\beta}>0$.

\item Let $\cR$ be of type $B$ or $C$. Then
\begin{equation}\label{wBC}
w_{\alpha\beta}w_{\beta\alpha}=\frac{1}{2}.
\end{equation}

\item Let $\cR=F_4$. or $G_2$.  Let $\alpha$ (resp. $\beta$) be a short
(resp. long) corner root of the Dynkin diagram associated to
$\cR$. Then,
\begin{align}\label{wf4} w_{\alpha\beta}& = m(\Phi)=\begin{cases} 2 &
    \text{if~~ $\cR=F_4$},\\
3 & \text{if~~ $\cR=G_2$}.
\end{cases}\\
 w_{\beta\alpha}& =1.
\end{align}

\end{enumerate}
\end{lemma}
\begin{proof}
The proof is explicit and case by case. We use the classification of
the root systems as given in \cite[Section 12.1]{H}. For the root
systems of type $C_r$, take as a base
\[ \alpha=\epsilon_1-\epsilon_2, \cdots, \epsilon_{r-1}-\epsilon_r, ~
2\epsilon_r=\beta.\]
The fundamental weights and coweights corresponding to the corner
roots are given by,
\[
\omega_{\alpha}^*=\epsilon_1, \quad \text{and} \quad
\omega_{\beta}^*=(\sum_{i=1}^r\epsilon_i)/2.\]
Further $\omega_{\alpha}=\omega_{\alpha}^*$ and
$\omega_{\beta}=2\omega_{\beta}^*$. Thus,
\[\begin{split}
 w_{\alpha\beta}w_{\beta\alpha} &=
<\omega_{\alpha}^*,\omega_{\beta}><\omega_{\beta}^*,\omega_{\alpha}>\\
&= \frac{1}{2}.
\end{split}
\]
For $F_4$, we take as a base,
\[\beta=\epsilon_2-\epsilon_3, ~\epsilon_3-\epsilon_4, ~\epsilon_4,
~(\epsilon_1-\epsilon_2 -\epsilon_3-\epsilon_4)/2=\alpha.\]
Then the fundamental coweights are given by,
\[
\omega_{\alpha}^*=2\epsilon_1, \quad \text{and} \quad
\omega_{\beta}^*=\epsilon_1+\epsilon_2.\]
The fundamental weights are given by
$\omega_{\alpha}=\omega_{\alpha}^*/2$ and
$\omega_{\beta}=\omega_{\beta}^*$. Hence,
\[\begin{split}
w_{\alpha\beta}& =<\omega_{\alpha}^*,\omega_{\beta}>=2,\\
w_{\beta\alpha}&=<\omega_{\beta}^*,\omega_{\alpha}>=1.
\end{split}
\]
For $G_2$, a base is given by,
 \[\alpha=\epsilon_1-\epsilon_2, ~\beta=-2\epsilon_1+\epsilon_2
 +\epsilon_3.\]
The fundamental coweights are given by,
\[
\omega_{\alpha}^*=-\epsilon_2+\epsilon_3, \quad \text{and} \quad
\omega_{\beta}^*=\frac{1}{3}(-\epsilon_1-\epsilon_2+2\epsilon_3).\]
The fundamental weights are given by
$\omega_{\alpha}=\omega_{\alpha}^*$ and
$\omega_{\beta}=3\omega_{\beta}^*$. Hence,
\[
\begin{split}
w_{\alpha\beta}& =<\omega_{\alpha}^*,\omega_{\beta}>=3,\\
w_{\beta\alpha}&=<\omega_{\beta}^*,\omega_{\alpha}>=1.
\end{split}
\]
This proves the lemma.
\end{proof}

\subsubsection{Proof of Proposition \ref{prop:nonmonic} for B and C}
We now prove Proposition \ref{prop:nonmonic} for the simple root
systems $\cR$ of type $B_r$ and $C_r$. We stick to the above
notation. We choose $\alpha$ as in Part (3) of  Lemma
\ref{lem:minus1}, and let $\beta$ be the other corner root of the
Dynkin diagram attached to $\cR$. Write
$\lambda=u\omega_{\alpha}+v\omega_{\beta}+d(\lambda)\rho$ as given by
Corollary \ref{cor:moniclambda}. By Lemmas \ref{lem:monic-trivial-est} and
\ref{lem:monic-sharp-est} we obtain the inequalities,
\begin{align} 2uw_{\beta\alpha}+1 & ~\geq ~ <\beta^*,
\lambda>~=~v+d(\lambda)\\
vw_{\alpha\beta}& ~\geq ~<\alpha^*, \lambda>
~=~u+d(\lambda).
\end{align} Consequently,
\[\begin{split} 2uw_{\beta\alpha}w_{\alpha\beta}& ~\geq~
(v+d(\lambda)-1)w_{\alpha\beta}\\ &~ \geq~
u+d(\lambda)+w_{\alpha\beta}(d(\lambda)-1).
\end{split}
\]
By Part (2) of Lemma
\ref{lem:walphabeta}, we have
$2w_{\beta\alpha}w_{\alpha\beta}=1$. Hence,
\[ 0~\geq~ d(\lambda)+w_{\alpha\beta}(d(\lambda)-1).\] But this
contradicts the positivity of $d(\lambda)$ and
$w_{\alpha\beta}$. Hence this proves Proposition
\ref{prop:nonmonic} for the simple root systems of type $B$ or $C$.

\subsubsection{Proof for $F_4$ and $G_2$}
For these simple root systems $\cR$, by Lemma \ref{lem:minus1},
the element $-1$ belongs to the Weyl
group  of $\cR$ as well as to that of $\cR_{\alpha}$ for any corner
root $\alpha$.  Let $\alpha$ denote the short
corner root and $\beta$ the long corner root. Write
$\lambda=u\omega_{\alpha}+v\omega_{\beta}+d\rho$, and we assume the
factorization $S(\lambda)=UVS(d(\lambda)\rho)$ is such that $U$ is
monic along $\alpha$ and $V$ is monic along $\beta$. Applying
the sharper estimates given by Lemma
\ref{lem:monic-sharp-est}, and from Part (3)  of Lemma
\ref{lem:walphabeta}, we get
\begin{align}
m(\Phi)v~=~ vw_{\alpha\beta}& ~\geq~ <\alpha^*, \lambda>~=~u+d,\\
u~=~ uw_{\beta\alpha}&~\geq~ <\beta^*, \lambda>~=~v+d.
\end{align}
If $\lambda$ as above does not satisfy the above equations, then any
invariant factorization is non-monic. This proves Proposition
\ref{prop:nonmonic}.

\begin{remark} Unfortunately, this proof develops a gap out here. If
  we try to look for cofactor expansions (say for $G_2$)
along any other linear form,
  we require that the top degree term for $S(\lambda)$ is not
  monic. This would require that the linear form is invariant under a
  subgroup of the Weyl group as above, and thus upto a Weyl translate
  is either $\omega_{\alpha}^*$ or $\omega_{\beta}^*$. Hence we cannot
  do better.
\end{remark}

\section{Proof of the uniqueness property} \label{sec:unique}
In this section, we give a proof of Theorem \ref{thm:unique} and
deduce some consequences. We begin with a couple of preliminary
lemmas.
\begin{lemma} \label{lem:leadcoeffs}
Suppose $\lambda, ~\mu$ are weights associated to a simple based root
system $\cR= (E, \Phi, \Delta)$, such that for two  distinct corner roots
$\alpha, ~\beta$ in the Dynkin diagram of $\cR$,
\[\lambda^{\alpha}=\mu^{\alpha}, \quad  \text{and} \quad \lambda^{\beta}=\mu^{\beta}.\]
Then $\lambda=\mu$.
\end{lemma}
\begin{proof}
The hypothesis for a particular corner root $\alpha$
implies an equality of the multiplicities $m_{\gamma}(\lambda)=m_{\gamma}(\mu)$
for all $\gamma\in \Delta$ not equal to $\alpha$. Since this happens  at
two corner roots, the lemma follows.
\end{proof}

We now give the proof of Theorem \ref{thm:unique}. We first
reformulate the hypothesis. By clearing the denominators, the
hypothesis can be reformulated as an equality of products of
Schur-Weyl elements, 
\begin{equation}\label{eqn:prodschurweyl}
 S(\lambda_1)S(\mu_2)=S(\lambda_2)S(\mu_1).
\end{equation}
We want to show that $\lambda_1=\lambda_2$ or $\mu_1$. 
Assuming that $\lambda_1\neq \mu_1$ (or equivalently, $\lambda_2\neq
\mu_2$),  this will show that
$\lambda_1=\lambda_2$ and $\mu_1=\mu_2$. 
The proof is by induction on the rank of the root system. Assume first
that $\Phi$ is of rank one isomorphic to the root system associated to
$sl(2)$.  The hypothesis indicates,
\[ (x^a-x^{-a})(x^e-x^{-e})~=~(x^b-x^{-b})(x^f-x^{-f}), \]
for some positive integers $a, ~b, ~e, ~f$, and $a\neq f$.
 The proposition follows immediately by comparing
the roots on both sides.

Now assume that $\cR$ is of rank $r$, and  the theorem has
been proved for all simple root systems of rank less than $r$. By
Lemma \ref{lem:leadcoeffs} and the inductive hypothesis, 
we can conclude the following: 

(i) Since
there are three corner roots for the simple root systems of type $D$
and $E$, the theorem follows for them. 

(ii) We can assume that there is a corner
root, say $\alpha$ at which
\begin{equation}\label{eqn:alpha}
 \lambda_1^{\alpha}=\lambda_2^{\alpha}\quad \text{and} \quad
\mu_1^{\alpha}=\mu_2^{\alpha},
\end{equation}
and another corner root $\beta$ where
\begin{equation}\label{eqn:beta}
 \lambda_1^{\beta}=\mu_1^{\beta}\quad \text{and} \quad
\lambda_2^{\beta}=\mu_2^{\beta}.
\end{equation}
Suppose $m_{\alpha}(\lambda_1)<m_{\alpha}(\lambda_2)$. 
We have, 
\begin{equation}
m_{\alpha}(\lambda_1)~=~m_{\alpha}(\mu_1)~<~m_{\alpha}(\lambda_2)~=~m_{\alpha}(\mu_2).
\end{equation}
Comparing the coefficients of the second leading term on both sides in
the cofactor expansion along $\alpha$, 
we get
\[ S((s_{\alpha}\lambda_1)^{\alpha})S(\mu_2^{\alpha})~=~
S(\lambda_2^{\alpha})S((s_{\alpha}\mu_1)^{\alpha}).\]
By induction, if $\lambda_2^{\alpha}=\mu_2^{\alpha}$, coupled with Equation
\ref{eqn:beta}, this implies $\lambda_2=\mu_2$ contradicting our
hypothesis that they are not equal. The other equality gives, 
\[
(s_{\alpha}\lambda_1)^{\alpha}=\lambda_2^{\alpha}=\lambda_1^{\alpha}.\]
But since $\lambda_1$ is regular, 
\[\lambda_1^{\alpha}=(s_{\alpha}\lambda_1)^{\alpha}=\lambda_1^{\alpha}+c\omega_{\alpha_n}^{\alpha},\]
for some positive integer $c$, where $\alpha_n$ is the unique root in the Dynkin diagram connected to
$\alpha$. This yields a contradiction and proves Theorem
\ref{thm:unique}.

\subsection{A coprimality property}

\begin{lemma}\label{lem:coprime}
Let $X$ be a finitely generated free abelian group and $a_1, ~a_2\in
X$. Then $e^{a_1}-1$ and $e^{a_2}-1$ are coprime elements in the group algebra
$\C[X]$, unless there exists integers $k, ~l$ different from zero such
that $ka_1=la_2$.
\end{lemma}
\begin{proof}
Consider the subspace generated by $a_1, ~a_2$ in the rational vector
space $X_{\Q}=X\otimes \Q$. Suppose they generate a two dimensional vector
subspace. Expand $a_1, ~a_2$ to a basis $\{a_1, a_2, \cdots, a_n\}$ of
$X_{\Q}$. Writing $x_i=e^{a_i}, ~i=1, \cdots, n$, we can identify
$\C[X_{\Q}]$  with the ring of fractional Laurent series in the variables
$x_1, \cdots, x_n$.
We claim that the elements   $x_1-1$ and $x_2-1$ are
coprime  in $\C[X_{\Q}]$. This follows from considering the degrees of
elements along each variable $x_i$, where the degree of an element
$U\in \C[X_{\Q}]$ along $x_i$ is defined as the
difference between the maximum  and minimum degrees in $x_i$ of the
various monomials occurring in $U$. It is clear that the degree of
$UV$ is the sum of the degrees of $U$ and $V$. From this it follows
that any element dividing both $x_1-1$ and $x_2-1$ has to be a
monomial.
\end{proof}

\begin{corollary} \label{cor:coprimeroots}
Let $\cR=(E, \Phi, \Delta)$ be a simple based root system, and
$\alpha\neq \pm\beta$ be two roots in $\Phi$. Then for any non-zero integers
$k, ~l$, the elements $e^{k\alpha}-1$ and $e^{l\beta}-1$ are coprime
in the group algebra $\C[P]$.

In particular, the element $S(d\rho)$ is a separable element in the
ring $\C[P]$ (see Part (c) of Proposition
\ref{prop:weyldenominator}).
\end{corollary}
\begin{proof} This follows from the previous lemma and the fact that
  if $\alpha$ and $\beta$ are rational multiples of each other in
  $P\otimes \Q$ precisely when $\alpha=\pm\beta$.
\end{proof}

In the following corollary, if the root system is simply laced, the notation
$\tilde{\rho}$ will denote $\rho$.

\begin{corollary}\label{cor:coprimeweyl}
Suppose $e, ~f$ are natural numbers with $(e,f)=d$. Assume further
that the multiplicity with which $m(\Phi)$ divides $e$ is not greater
than the multiplicity with which it divides $f$. Then the elements
$S(e\rho)/S(d\rho)$ and $S(f\tilde{\rho})/S(d\rho)$ are coprime in the
ring $\C[P]$.
\end{corollary}
\begin{proof} By hypothesis,
  since $(e,f)=d$, it follows that $(e,m(\Phi)f)=d$.
By the previous corollary, we need to bother only with
  the individual factors associated to a root $\alpha$.
It is clear that the elements
  $(e^{2e\alpha}-1)/(e^{2d\alpha}-1)$ and
  $(e^{2fm(\Phi)\alpha}-1)/(e^{2d\alpha}-1)$ are coprime in
  $\C[P]$.
\end{proof}

The above corollary combined with the uniqueness Theorem
\ref{thm:unique} gives us the following:
\begin{corollary}\label{cor:coprime}
   Let $\cR=(E, \Phi, \Delta)$ be a simple based root
system, and assume that Theorem \ref{thm:main} holds for $\cR$.
Suppose $\lambda\neq \mu$ are
dominant regular weights belonging to $m(\Phi)^iP$ such that at most
one of them is in the space $m(\Phi)^{i+1}P^*$.  Let $d$ be the greatest
common divisor of $d(\lambda)$ and $d(\mu)$. Then $S(\lambda)/S(d\rho)$
and  $S(\mu)/S(d\rho)$ are coprime.
\end{corollary}
\begin{proof} By Theorem \ref{thm:main}, the factors $C(\lambda)$ and
  $C(\mu)$ if
  non-trivial are irreducible. If they coincide, then by the
  uniqueness Theorem \ref{thm:unique}, $\lambda=\mu$. Hence we are
  left with showing the Weyl denominator type of  $S(\lambda)/S(d\rho)$
and  $S(\mu)/S(d\rho)$ are coprime. Suppose $\lambda\in
m(\Phi)^iP\backslash m(\Phi)^{i+1}P^*$. Then
$d(\lambda)=m(\Phi)^id'(\lambda)$ with $d'(\lambda)$ coprime to
$m(\Phi)$. The corollary follows from Corollary \ref{cor:coprimeweyl}.
\end{proof}

Finally, we deduce a coprimality result needed in the proof of
Proposition \ref{prop:nonmonic}:
\begin{proposition}\label{prop:coprime}
  Let $\cR=(E, \Phi, \Delta)$ be a simple based root
system of rank $l$  and assume that Theorem \ref{thm:main} holds for
any simple root system of rank less than $l$. Suppose $\eta\in
m(\Phi)^iP\backslash m(\Phi)^{i+1}P^*$. Let $\alpha$ be a corner root
in the Dynkin diagram associated to $\cR$. Then the elements
$S(\eta^{\alpha})/S(d(\eta)\rho^{\alpha})$
and  $S((s_{\alpha}\eta)^{\alpha})/S(d(\eta)\rho^{\alpha})$ are
coprime in the ring $\C[P_{\alpha}]$.
\end{proposition}
\begin{proof}
 The multiplicity   at a fundamental root
$\gamma \in \Delta_{\alpha}$ of a weight $\mu^{\alpha}$ differs from that of
$\mu$ only at the root $\alpha_n$ connected to $\alpha$ in the
Dynkin diagram.  Suppose  $\eta^{\alpha}$ or $(s_{\alpha}\eta)^{\alpha}$ belongs to
$m(\Phi)^{i+1}P_{\alpha}^*$, where we have assumed $\cR_{\alpha}$ is
not simply laced.
From the hypothesis this can happen only when both $\alpha_n$
and  $\alpha$ are short roots. Hence,
\[ (s_{\alpha}\eta)^{\alpha}=\eta^{\alpha}+m_{\alpha}(\lambda).\]
The hypothesis on $\eta$ implies that at most one of
 $\eta^{\alpha}$ or $(s_{\alpha}\eta)^{\alpha}$ belongs to
 $m(\Phi)^{i+1}P_{\alpha}^*$.
The proposition now follows from Corollary \ref{cor:coprime}.
\end{proof}

\subsection{Unique decomposition of tensor products} We now indicate a
proof of Theorem \ref{thm:tensor} stated in the introductory section,
and our original motivation for establishing an irreducibility
result. Unfortunately as the proof of Theorem \ref{thm:main} that we
have presented out here has a gap,  we need to assume {\em
  Assumption NMFG}.

Suppose we have an isomorphism of tensor products
\[ V_1\otimes \cdots\otimes V_n\simeq W_1\otimes
\cdots\otimes W_m,\]
as in the hypothesis of Theorem \ref{thm:tensor}. Let $(\lambda_i)$
(resp. $\mu_j$) denote the highest weight of the irreducible
representation $V_i$ (resp. $W_j$). The above hypothesis translates to
an equality of products of characters:
\[\prod_{i=1}^n S(\lambda_i+\rho)/S(\rho)=\prod_{j=1}^m
S(\mu_j+\rho)/S(\rho).\]
By the uniqueness Theorem \ref{thm:unique} and the irreducibility
Theorem \ref{thm:main}, if one of the factors, say $\lambda_1+\rho$ is
not a multiple of $\rho$, then there exists a $j$ such that
$\lambda_1+\rho=\mu_j+\rho$. Cancelling these factors, we are left
with an equality involving fewer characters and we are done.

Hence, we are reduced to the case that each of these weights
$\lambda_i+\rho$ and $\mu_j+\rho$ ($i=1, \cdots, n,
~j=1, \cdots, m$) is a multiple of either $\rho$ or $\tilde{\rho}$.
By the coprimality result Corollary \ref{cor:coprimeroots},
we have the following
equalities for any root $\alpha$:
\[\prod_{i=1}^n\frac{e^{d_i/2\alpha}-e^{-d_i/2\alpha}}{e^{\alpha/2}-e^{-\alpha/2}}=\prod_{j=1}^m\frac{e^{e_j/2\alpha}-e^{-e_j/2\alpha}}{e^{\alpha/2}-e^{-\alpha/2}},\]
for some natural numbers $d_i, ~e_j$ (depending also on the relative
length of the root).
An easy argument using roots of these expressions, establishes Theorem
\ref{thm:tensor}.

\section{Reduction to the invariant case for $GL(r)$}\label{sec:noninv}
In this section our aim is to extend the irreducibility result for
$C(\lambda)$ in the ring $\C[P]^W$ given by Theorem \ref{thm:main} to
the larger ring $\C[P]$, when we are working with $GL(r)$,.i.e., to
show the irreducibility of $C(\lambda)$ in the polynomial ring $C[x_1,
\cdots, x_r]$, as claimed in Theorem \ref{thm:glrirr}.

Suppose $\lambda$ is a dominant integral weight for $GL(r)$ which is
not a multiple of $\rho$. By Lemma \ref{lem:leadcoeffs}, there is at
least one corner root, say  $\alpha$,  such that the leading coefficient
 of
$C(\lambda)$ along $\alpha$ is not a unit (for $GL(r)$, this is clear:
if $C(\lambda)$ is monic as a polynomial in $x_1$, then being
symmetric the constant term is of the form
$x_2^{a_1-rd(\lambda)}+\cdots +x_r^{a_1-rd(\lambda)}$).

Suppose $\lambda^{\alpha}$ is not a multiple of $\rho^{\alpha}$. Then
there exists a smallest factor $U$ of $C(\lambda)$ in the ring $C[x_1,
\cdots, x_r]$, such that its leading coefficient along $\alpha$ is
divisible by the irreducible component $C(\lambda)^{\alpha}$ of the
leading coefficient of   $C(\lambda)$ along $\alpha$. By construction,
$U$ is irreducible, and since the polynomials $S(\mu)$ are separable
for any regular weight $\mu$, $U$ will also be
$W_{\alpha}$-invariant.

If $V=C(\lambda)/U$ is monic, then by Lemma \ref{lem:glrmonicnonsym}
proved below, both $U$ and $V$ will be symmetric, and the
irreducibility result follows by Theorem \ref{thm:main}.

On the other hand, if the leading coefficient of $V$ along $\alpha$ is
not a unit, then just as in reduction part of the proof of Theorem
\ref{thm:main} to that of Proposition \ref{prop:key},  we have reduced
the irreducibility statement to Proposition \ref{prop:key} for lower
rank.

Hence, we have reduced to the case that $\lambda^{\alpha}$ is a
multiple of $\rho^{\alpha}$ for any corner root  $\alpha$. If $r\geq
4$, the rank is at least $3$, and this implies that $\lambda$ is a
multiple of $\rho$ contradicting our hypthesis on $\lambda$.

Thus we are in the $GL(3)$ situation: if the factor $U$ is irreducible
and the quotient $V$ as above is monic, then we are through
 by  Lemma \ref{lem:glrmonicnonsym} as argued above. Otherwise, the
 leading coefficients of both $U$ and $V$ along $\alpha$ are not
 units. But in this case, we observe that Proposition \ref{prop:arith}
or Corollary \ref{cor:key-sl2} is valid
 without any assumptions of symmetry, and hence the irreduciblity
follows as in the reduction of the proof of the main Theorem
\ref{thm:main} to Proposition \ref{prop:key}.

This proves Theorem \ref{thm:glrirr}, modulo the following lemma:

\begin{lemma} \label{lem:glrmonicnonsym} With the above notation, let
  $C(\lambda)=UV$ be a factorization of $C(\lambda)$ in the ring
$C[x_1,\cdots, x_r]$.
Assume further the following: $U$ is irreducible, and $V$ is either
monic or the constant coeffecient  $V_0$ of $V$ is a monomial
(considered as a polynomial in $x_1$).  Then $U$  and $V$
are symmetric polynomials.
\end{lemma}
\begin{proof} We first observe that if  both the leading and constant
coeffecients of $C(\lambda, d(\lambda))$ are trivial.  then $\lambda^{(1)}=
\lambda^{(r-1)}=d\rho$. This implies $\lambda=d\rho$  and so  $C(\lambda,
d(\lambda))=1$. Hence assume that the leading coeffecient of
$C(\lambda, d(\lambda))$ is not monic, and that $V$ is monic.
From the proof of Proposition \ref{prop:nonmonicglr}, we get
\[ u\geq a_1-a_2.\] Since the leading coeffecients of $U$ and
$C(\lambda, d(\lambda))$ match, and $U$ is irreducible, we see that
$U$ is fixed by the subgroup $S_r(1)$ of the symmetric group $S_r$ of
permutations of the set ${1,\cdots, r}$ fixing the element $1$. Let
$\sigma$ be the transposition in $S_r$ interchaning $1$ and $2$. If
$\sigma$ fixes $U$, then since the group generated by $\sigma$ and
$S_r(1)$ is $S_r$, we conclude that $U$ is symmetric.  If $U$ is not
symmetric, then  since $U$ is irreducible, $U^{\sigma}$ must divide
$V$. Now the degree in $x_2$ variable of $U$ is at least $a_2-d$.
Hence,
\[v\geq \mbox{deg}_{x_1}U^{\sigma}\geq \mbox{deg}_{x_2}U\geq a_2-d.\]
This yields,
\[\mbox{deg}_{x_1}(S(\lambda))=a_1=u+v+2d\geq a_1-a_2+a_2-d+2d>a_1,\]
a contradiction.
\end{proof}

\begin{center}
{\bf Acknowledgement}
\end{center}
I sincerely thank P. Deligne for indicating the extra reducibility
that happens for the non-simply laced simple Lie algebras arising from
duality (as soon as I stated the conjecture for $GL(r)$).
This nipped the problem in the bud  of
figuring out this extra reducibility in the general case.

Most of this work was done when I was visiting MPIM at Bonn during
2005-06.
Apart from streamlining the proofs, especially the
inductive proof of Proposition
\ref{prop:key}, the main new ingredient is the reduction to the
symmetric case given in Section \ref{sec:noninv} which was obtained
when I was visiting Muenster during the summer of 2008, supported
by SFB 478.  My sincere
thanks to these institutions, to IAS at Princeton where I
visited during the first half of 2007 and to ICTP, Trieste (October
2011) for their excellent hospitality and working environment while I
worked on this and other questions.


\begin{thebibliography}{JPSH}
\bibitem[B]{B} N. Bourbaki, {\em Lie groups and Lie algebras}.

\bibitem[DZ]{DZ} R. Dvornicich and U. Zannier,  {\it Newton functions
generating symmetric fields and irreducibility of Schur polynomials},
Adv. Math. 222 (2009), no. 6, 1982-2003.

\bibitem[FI]{FI} R. Fossum and B. Iversen, {\it On Picard groups of
algebraic fibre spaces},  J. Pure Appl. Algebra 3 (1973), 269-280.

\bibitem[FH]{FH} W. Fulton and J. Harris, {\em Representation theory:
a first course}, Graduate texts in mathematics, 129,
Springer-Verlag, New York, 1991.

\bibitem[H]{H} J. E. Humphreys, {\it Introduction to Lie algebras and
representation theory}, Springer-Verlag, New York, 1972.

\bibitem[KKLV1]{KKLV1}  F. Knop, H. Kraft  and T. Vust, {\it The Picard
group of a $G$-variety} in Algebraic transformation groups and
invariant theory, 77-87 eds. H. Kraft, P. Slowdowy and T. A. Springer,
DMV Seminar, Band 13, Birkhauser-Verlag,  Berlin, 1989.

\bibitem[KKLV2]{KKLV2}  F. Knop, H. Kraft, D. Luna and T. Vust, {\it
Local properties of algebraic  group actions} in Algebraic
transformation groups and invariant theory, 63-75 eds. H. Kraft,
P. Slowdowy and T. A. Springer,  DMV Seminar, Band 13,
Birkhauser-Verlag,  Berlin, 1989.

\bibitem[R]{R} C. S. Rajan, {\em Unique decomposition of tensor
products of irreducible representations of simple algebraic groups},
Annals of Mathematics, 160, (2004), 683-704.



\end{thebibliography}
\end{document}